\documentclass{article}
\usepackage[utf8]{inputenc}

\usepackage{fullpage}
\usepackage{amsfonts}
\usepackage{amsmath}
\usepackage{amssymb}
\usepackage{amsthm}

\usepackage{bm}
\usepackage{graphics}
\usepackage{graphicx}

\usepackage{hyperref}

\usepackage{subcaption}

\usepackage{latexsym,amsfonts,amssymb,amsmath,amsthm}

\usepackage{booktabs} 
\usepackage{graphicx}
\usepackage{pgfplots}
\usepackage[all]{nowidow}
\usepackage{soul}
\usepackage[utf8]{inputenc}
\usepackage{tikz}
\usetikzlibrary{automata}
\usepackage{multicol,comment}
\usepackage{algpseudocode,algorithm,algorithmicx}

\usepackage{cite}
\usepackage{hyperref}
\usepackage[inline]{enumitem} 

\usepackage{ulem}

\definecolor{blue}{HTML}{1F77B4}
\definecolor{orange}{HTML}{FF7F0E}
\definecolor{green}{HTML}{2CA02C}

\pgfplotsset{compat=1.14}

\newtheorem{tm}{Theorem}[section]
\newtheorem{prop}[tm]{Proposition}

\newtheorem{lem}[tm]{Lemma}

\newtheorem{rk}[tm]{Remark}

\numberwithin{equation}{section}
\numberwithin{tm}{section}

\usepackage[makeroom]{cancel}

\title{On the dynamics of an epidemic patch model with mass-action transmission mechanism and asymmetric dispersal patterns}

\author{  Rachidi B. Salako\footnote{rachidi.salako@unlv.edu; Department of Mathematical Sciences, University of Nevada Las Vegas, Las Vegas, USA} \quad and \quad Yixiang Wu\footnote{yixiang.wu@mtsu.edu; Department of Mathematical Sciences, Middle Tennessee State University,
 Murfreesboro, Tennessee 37132, USA} }
\date{}

\begin{document}

\maketitle

\begin{abstract} 
  This paper  examines an epidemic patch model with mass-action transmission mechanism and asymmetric connectivity matrix. Results on the global dynamics of  solutions  and the spatial structures of endemic solutions are obtained. In particular, we show that when the basic reproduction number $\mathcal{R}_0$ is less than one and the dispersal rate of the susceptible population $d_S$ is large, the population would eventually stabilize at the disease-free equilibrium. However,  the disease may persist if $d_S$ is small, even if $\mathcal{R}_0<1$.   In such a scenario, explicit conditions on the model parameters that lead to the existence of multiple endemic equilibria are identified. These results  provide new insights into the dynamics of infectious diseases in multi-patch environments. Moreover,   results in \cite{li2023sis}, which is for the same model but with symmetric connectivity matrix, are generalized and improved. 
\end{abstract}


\smallskip


\section{Introduction} 
Differential equations have long been used to describe the transmission of infectious diseases \cite{1991book,  2008book, brauer2019mathematical, Diekmann2000, martcheva2015introduction}. To model the spatial spread of diseases, various patch (or metapopulation) models have been proposed and analyzed. For instance, Hethcote \cite{hethcote1976qualitative} analyzed the global stability of a two-patch susceptible-infected-susceptible (SIS)   epidemic patch model; Arino and van den Driessche \cite{arino2003multi} computed the basic reproduction number of a multi-city epidemic model; Wang and Zhao \cite{wang2004epidemic} used an SIS patch model to illustrate the importance of dispersal in the spread of disease; Guo, Li and Shuai \cite{guo2006global} developed a graph-theoretical method to construct Lyapunov functions and  show the stability of the endemic equilibrium of multigroup epidemic models. We refer the interested readers to the survey papers \cite{wang2007epidemic,arino2009diseases} for more  works on epidemic patch models. 

In the pioneering work \cite{allen2007asymptotic}, Allen \textit{et al}.  proposed the following SIS   epidemic patch model with standard-incidence transmission mechanism to study the impact of environmental heterogeneity and population movement on the spread of an infectious disease:
\begin{equation}\label{model-si}
\begin{cases}
\displaystyle\frac{d S_i}{dt}=d_S\sum_{j\in\Omega, j\neq i}(L_{ij} S_j-L_{ji} S_i)-\beta_{i} \frac{S_iI_i}{S_i+I_i}+\gamma_{i} I_{i},  & i\in\Omega,\ t>0, \cr 
\displaystyle\frac{d I_{i}}{dt}=d_I\sum_{j\in\Omega, j\neq i}(L_{ij} I_{j}-L_{ji} I_{i})+\beta_{i} \frac{S_iI_i}{S_i+I_i}-\gamma_{i} I_{i}, & i\in\Omega,\ t>0. 
\end{cases}
\end{equation}
Here, the population is supposed to   live in patches (e.g., cities, countries), and $\Omega=\{1, 2, \dots, n\}$ is the collection of patch labels, where $n\ge 2$; $S_i$ and $I_i$ are the number of susceptible and infected people in patch $i$, respectively; $\beta_i$ and $\gamma_i$ are disease transmission and recovery rates, respectively; $L_{ij}$, $i\neq j$, is the degree of movement of individuals from patch $j$ to patch $i$; $d_S$ and $d_I$ are the dispersal rates of susceptible and infected individuals, respectively. In \cite{allen2007asymptotic}, the authors defined a basic reproduction number $\mathcal{R}_0$ and showed that $\mathcal{R}_0$ serves as a threshold parameter: the disease free equilibrium (DFE) is globally asymptotically stable if $\mathcal{R}_0<1$, and the model has a unique endemic equilibrium (EE) if $\mathcal{R}_0>1$.
Most importantly, they showed that the disease component of the endemic equilibrium approaches zero if $d_S\to 0$ under the assumption that  $\{i\in\Omega: \beta_i-\gamma_i<0\}$ and $\{i\in\Omega: \beta_i-\gamma_i>0\}$ are non-empty. Biologically, this suggests that one may control the disease by limiting the movement of susceptible people. Later, Li and Peng \cite{li2019dynamics} studied the asymptotic profiles of the EE as $d_S$ and/or $d_I$ approach zero; Chen \textit{et al.} \cite{chen2020asymptotic} revisited model \eqref{model-si} without assuming that the connectivity matrix $L$ is symmetric; the monotonicity of $\mathcal{R}_0$ in $d_I$ has been studied in \cite{gao2020does,gao2020fast, chen2020asymptotic}; The impact of dispersal rates on the number of infected cases was investigated in \cite{gao2020does,gao2021impact}.

In this paper, we consider the following   epidemic patch model with mass-action transmission mechanism: 
\begin{equation*}
\begin{cases}
\displaystyle\frac{d S_i}{dt}=d_S\sum_{j\in\Omega, j\neq i}(L_{ij} S_j-L_{ji} S_i)-\beta_{i} {S_iI_i}+\gamma_{i} I_{i},  & i\in\Omega,\ t>0, \cr 
\displaystyle\frac{d I_{i}}{dt}=d_I\sum_{j\in\Omega, j\neq i}(L_{ij} I_{j}-L_{ji} I_{i})+\beta_{i} {S_iI_i}-\gamma_{i} I_{i}, & i\in\Omega,\ t>0.
\end{cases}
\end{equation*}
All the parameters and variables carry the same biological meaning with model \eqref{model-si} except that the standard-incidence transmission mechanism $\beta_i S_iI_i/(S_i+I_i)$ is replaced by the mass-action transmission mechanism $\beta_i S_i I_i$. 
For simplicity, let $L_{ii}=-\sum_{j\in\Omega, j\neq i}L_{ji}$, which is the total degree of movement out from patch $i$. Then the model can be written as follows:
\begin{equation}\label{model-mass-action}
\begin{cases}
\displaystyle\frac{d S_i}{dt}=d_S\sum_{j\in\Omega}L_{ij} S_j-\beta_{i} {S_iI_i}+\gamma_{i} I_{i},  & i\in\Omega,\ t>0, \cr 
\displaystyle\frac{d I_{i}}{dt}=d_I\sum_{j\in\Omega}L_{ij} I_{j}+\beta_{i} {S_iI_i}-\gamma_{i} I_{i}, & i\in\Omega,\ t>0.
\end{cases}
\end{equation}
This model has been studied in \cite{li2023sis} when the connectivity matrix $\mathcal{L}=(L_{ij})$ is symmetric. In particular, the authors defined a basic reproduction number $\mathcal{R}_0$ and studied the global dynamics of the model if $d_S=d_I$ or $(\beta_1, \dots, \beta_n)$ is a multiple of $(\gamma_1, \dots,\gamma_n)$. They proved the existence of  EE when $\mathcal{R}_0>1$ and showed that it was unique with some additional assumptions, and they  studied the asymptotic profiles of the EE as $d_S$ and/or $d_I$ approach zero.

If \eqref{model-mass-action} has only one patch, then it is the following classic ordinary differential equation epidemic model by Kermack and McKendrick \cite{kermack1927contribution}:
\begin{equation} \label{1-patch-model}
\begin{cases}
S'=-\beta SI+\gamma  I,\\
I'=\beta SI-\gamma  I,\\
S+I=N.
\end{cases}
\end{equation}
For this model, one can define a basic reproduction number $R_0=N\beta/\gamma$, which is the average number of cases one infected case generates after it is introduced to a population without the disease. The basic reproduction number of \eqref{1-patch-model} serves as a threshold value as it does for many epidemic models: if $R_0<1$, the  DFE $(N,0)$ of \eqref{1-patch-model} is globally asymptotically stable; if  $R_0>1$, the model \eqref{1-patch-model} has a unique globally stable  EE $(\gamma/\beta, N-\gamma/\beta)$. We may expect that such a threshold result is true for the patch model \eqref{model-mass-action}. However, we will show that it is possible for \eqref{model-mass-action} to have multiple EEs when $\mathcal{R}_0<1$ under certain conditions on the parameters. Our results indicate that the disease dynamics can be more complicated in a patchy environment.

The objective of the current paper is two fold. Firstly, we will obtain and improve the results in \cite{li2023sis}  without the assumption that the matrix $L$ is symmetric. If $L$ is not symmetric, many techniques (e.g., construction of  Lyapunov functions)  adopted in \cite{li2023sis} cannot be carried over. So we will develop novel techniques to overcome the asymmetry of $L$, which may be applied to study similar patch models. Some of our global stability results on the DFE seem to be new even in the symmetric case. Secondly, we will study the uniqueness/non-uniqueness of the EE solutions. As shown in Theorem \ref{theorem-ds}, using $d_S$ as a bifurcation parameter, it is possible that when $\mathcal{R}_0<1$ the model has multiple EE solutions if $d_S$ is small, while there is no EE solutions and the DFE is globally asymptotially stable if $d_S$ is large. Our result also suggests that models with mass-action mechanism may have very different behavior compared with models with  standard incidence mechanism (in model \ref{model-si}, the DFE is always globally asymptotically stable when $\mathcal{R}_0<1$). We refer the readers to the last section for  more discussions about the novelty and implications of our results. 

We remark that there are many recent works on reaction-diffusion epidemic models that are related to models \eqref{model-si} and \eqref{model-mass-action}. In particular, the reaction-diffusion counterpart of \eqref{model-si} has been analyzed in \cite{Allen,CuiLamLou,CuiLou,KuoPeng,Li2018, li2020dynamics,LouSalako2021,LSS2023,peng2021global,Peng2013, PengZhao,Tuncer2012}, while the reaction-diffusion version of \eqref{model-mass-action} has been studied in \cite{CS2023b, DengWu, salako2023, WuZou,tao2023analysis,wen2018asymptotic,castellano2022effect,peng2023novel}. Although the results for the patch and reaction-diffusion models share many similarities, the techniques for the patch models especially when the connectivity matrix is not symmetric can be very different from those for the reaction-diffusion models.  

The rest of our paper is organized as follows. In section 2, we introduce the notations and present some preliminary results. In section 3, we study the global stability of the model and the existence and uniqueness/non-uniqueness of the EE solutions. In section 4, we study the asymptotic profiles of the EE solutions as $d_S$ and/or $d_I$ approach zero. In section 5, we give some concluding remarks. 

\section{Preliminaries}\label{Sec2}

 \subsection{Notations, assumptions, and definitions} 
Throughout the paper, given a real number $x$, we use the conventional notations $x_+=\max\{x,0\}$ and $x_{-}:=\max\{0,-x\}$ so that $x=x_+-x_{-}$. A bold letter always represents  a column vector in $\mathbb{R}^n$, and its no-bold form with a numeric subscript will be a component of it. For example, for any $\bm X\in\mathbb{R}^n$, one has $\bm X=(X_1, \dots, X_n)^T$, where $X_j\in\mathbb{R}$ for  $j\in\Omega$. 
For convenience, we write $(0, \dots, 0)^T$ as ${\bm 0}$ and $(1, \dots, 1)^T$ as ${\bm 1}$. For $\bm X\in\mathbb{R}^n$, define 
 $$
{\bm X}_{m}:=\min_{j=1,\cdots, n} X_j,\quad  {\bm X}_{M}:=\max_{j=1\cdots,n} X_j,\quad 
 \|\bm X\|_1:=\sum_{j=1}^n|X_{i}|,\quad \text{and}\quad \|\bm X\|_{\infty}:=\max_{j\in\Omega}| X_j|.
$$
We denote by ${\rm diag}(\bm X)$ the diagonal matrix with diagonal entries $[{\rm diag}(\bm X)]_{ii}= X_i$ for all $i=1,\cdots,n$. We write $\bm X\ge  \bm 0$ if $X_i\ge  0$ for all $i\in\Omega$, $\bm X\gg  \bm 0$ if $X_i>  0$ for all $i\in\Omega$, and $\bm X>\bm 0$ if $\bm X\ge \bm 0$ and $\bm X\neq \bm 0$. 
 We denote by $<\cdot,\cdot>$ the inner product on $\mathbb{R}^n$, that is 
 $$ 
 \big<\bm X,\bm Y\big>=\sum_{i=1}^n X_i Y_i,\quad \forall\ \bm X,\ \bm Y\in\mathbb{R}^n.
 $$
 For  $\bm X, \bm Y\in \mathbb{R}^n$, define the Hadamard product $ \bm X\circ \bm Y :=(X_1Y_1,\cdots, X_n Y_n)^T$ 
 and set  ${\bm X}/{\bm Y}=(X_1/Y_1,\cdots,$
 $ X_n/ Y_n)^T$ if $Y_i\neq 0$ for all $i\in\Omega$.
Adopting these notations, system \eqref{model-mass-action} can be rewritten in the compact form
 \begin{equation*}
     \begin{cases}
         \bm S'=d_S\mathcal{L}\bm S-\bm\beta\circ\bm I\circ\bm S+\bm\gamma\circ\bm I, & t>0,\cr 
        \bm I'=d_I\mathcal{L}\bm I+\bm\beta\circ\bm I\circ \bm S-\bm\gamma\circ\bm I, & t>0.
     \end{cases}
 \end{equation*}

Throughout this work, we make the following assumptions:
\begin{itemize}
\item[{\bf (A1)}] The matrix $\mathcal{L}=(L_{i,j})_{i,j=1}^K$ is quasi-positive (i.e., $L_{ij}\ge 0$ for any $i\neq j$)  and irreducible.
\item[{\bf (A2)}]  $\bm S^0, \bm I^0\ge \bm 0$,  and   $\sum_{j\in\Omega}( S^{0}_j+ I^{0}_j):=N$, where $N$ is some fixed positive constant.
\item[{\bf (A3)}] $\bm\beta, \bm\gamma\gg \bm 0$ and $d_S, d_I>0$.
\end{itemize}
Due to biological interpretations of the vectors $\bm S$ and $\bm I$, we will only be interested in nonnegative solutions of \eqref{model-mass-action}. Let $\mathbb{R}_+$ denote the set of nonnegative real numbers. For any initial data $(\bm S(0), \bm I(0))=(\bm S^0, \bm I^0)\in\mathbb{R}_+^n\times \mathbb{R}_+^n$,  \eqref{model-mass-action} has a unique nonnegative solution $(\bm S(t),\bm I(t))$ defined on a maximal interval of existence $[0,T_{\max})$. Summing up all the equations in \eqref{model-mass-action}, we find that $\frac{d}{dt}\sum_{i\in\Omega} (S_i+I_i)=0$, which means that the total population is preserved. Therefore, for any initial data $(\bm S^0, \bm I^0)$ satisfying {\bf (A2)}, the solution satisfies 
\begin{equation}\label{Eq1:1}
 \sum_{j\in\Omega}( S_{j}(t)+ I_{j}(t))=N,\quad \forall\ t>0.
\end{equation}
This means that the solution of  \eqref{model-mass-action} exists globally and $T_{\max}=\infty$.  Thanks to \eqref{Eq1:1} and the fact that the positive constant $N$ is fixed, the semiflow generated by  solutions of \eqref{model-mass-action} leaves invariant the set 
$$
\mathcal{E}:=\Big\{(\bm S,\bm I)\in \mathbb{R}_+^n\times \mathbb{R}_+^n\ : \ \sum_{j\in\Omega}( S_j+ I_{j})=N\Big\}.
$$  
It is easy to see that if $\bm I^0=\bm 0$ then $\bm I(t)=\bm 0$ for all $t\ge 0$. Assumption {\bf (A1)} indicates that the patches are strongly connected, i.e. for any two patches the individuals can move directly or indirectly from one patch to another.  By {\bf (A1)} and the comparison principle, if $\bm I^0>\bm 0$,  then $\bm S(t)\gg \bm 0$ and $\bm I(t)\gg \bm 0$ for all $t>0$.

For  $n\times n$ real-valued square matrix $M$, let $\sigma(M)$ be the set of eigenvalues of $M$,  $\sigma_{*}(M)$ be the spectral bound, i.e.,
\begin{equation*}
    \sigma_{*}(M):=\max\{\mathfrak{R}e(\lambda)\ :\ \lambda\in\sigma(M)\},
\end{equation*}
where $\mathfrak{R}e(\lambda)$ is the real part of $\lambda\in \mathbb{C}$, 
and $\rho(M)$ be the spectral radius, i.e.,
$$
\rho(M):=\max\{|\lambda|\ :\ \lambda\in\sigma(M)\}.
$$

Since $\mathcal{L}$ is quasi-positive and irreducible, it generates a strongly-positive matrix-semigroup $\{e^{t\mathcal{L}}\}_{t\ge 0}$ on $\mathbb{R}^n$.  Moreover since $\sum_{i\in\Omega}\mathcal{L}_{ij}=0$ for each $j\in\Omega$,  by the Perron-Frobenius theorem, $\sigma_*(\mathcal{L})=0$ is a simple eigenvalue of $\mathcal{L}$. Furthermore, there is an   eigenvector $\bm\alpha$ associated with $\sigma_*(\mathcal{L})$ satisfying 
\begin{equation}\label{alpha-eq}
\mathcal{L}\bm\alpha=\bm 0,\quad \sum_{j\in\Omega}{\alpha}_j=1,\quad \text{and} \quad {\alpha}_j>0,\ \forall \ j\in\Omega,
\end{equation}  
and $\bm\alpha$ is the unique nonnegative eigenvalue of $\mathcal{L}$ with $\sum_{j\in\Omega}{\alpha}_j=1$. Throughout the paper, $\bm\alpha$ is fixed and satisfies \eqref{alpha-eq}.
 
An equilibrium solution  $(\bm S, \bm I)$ of  \eqref{model-mass-action} is called a \textit{disease free equilibrium} (DFE) if $\bm I=\bm 0$ and an \textit{endemic equilibrium} (EE) if $\bm I\neq \bm 0$.  It is easy to check that $(N\bm\alpha, \bm 0)$ is the unique DFE. To define the basic reproduction number of  \eqref{model-mass-action}, we linearize the model at DFE and find that the stability of DFE is determined by $\sigma_*(d_I\mathcal{L}+{\rm diag}(N\bm\alpha\circ\bm\beta-\bm\gamma))$. So define 
 \begin{equation}\label{F}
 {F}={\rm diag}(N\bm\alpha\circ \bm\beta) \quad \text{and}\quad {V}:={\rm diag}(\bm \gamma)-d_I\mathcal{L}.
 \end{equation} 
Then  ${V}$ is invertible. Following the next generation matrix approach \cite{Driessche, Diekmann}, the \textit{basic reproduction number} $\mathcal{R}_{0}$ of \eqref{model-mass-action}  is given by
 \begin{equation}\label{R-0}
      \mathcal{R}_{0}:=\rho({F}{V}^{-1}).
 \end{equation}
It is well known that $\mathcal{R}_{0}-1$ and $\sigma_*(-V+F)=\sigma_*(d_I\mathcal{L}+{\rm diag}(N\bm\alpha\circ\bm\beta-\bm\gamma))$ have the same sign, and the DFE is locally asymptotically stable if  $\mathcal{R}_{0}<1$ and unstable if  $\mathcal{R}_{0}>1$ \cite{Driessche}.  By \cite{chen2020asymptotic,gao2020fast}, $\mathcal{R}_{0}$ is  decreasing in $d_I$ with 
\begin{equation}\label{R-0-limit}
\lim_{d_I\to 0}\mathcal{R}_0=\max_{j\in\Omega}\frac{N\alpha_j\beta_j}{\gamma_j}\ \ \text{and}\ \ \lim_{d_I\to \infty}\mathcal{R}_0=\frac{\sum_{j\in\Omega}N\alpha_j^2\beta_j}{\sum_{j\in\Omega}\alpha_j\gamma_j}.
\end{equation}
Moreover, $\mathcal{R}_0$ is strictly decreasing in $d_I$ if and only if $\bm\gamma$ is not a multiple of $\bm\alpha\circ\bm\beta$.

 For convenience, we define $\bm r=(r_1, \dots, r_n)^T\in\mathbb{R}^n$ with  
 \begin{equation*}
    { r}_{j}
    :=\frac{ \gamma_{j}}{ \beta_{j}}, \quad \forall\; j\in\Omega.
\end{equation*}
Biologically, $r_j$ measures the risk of patch $j$ when the patches are isolated, i.e. $d_S=d_I=0$.

\subsection{Preliminary results}
Recall that   $\bm \alpha$ is the eigenvector associated with $\sigma_*(\mathcal{L})=0$ as described in \eqref{alpha-eq}. By classical results in matrix theory, there is a subspace $\mathcal{Z}\subset\mathbb{R}^n$ invariant under $\mathcal{L}$ such that 
$$
\mathbb{R}^n={\rm span}(\bm \alpha)\oplus\mathcal{Z}.
$$
So there is projection map $P : \mathbb{R}^n\to {\rm span}(\bm \alpha)$ such that $P\bm \alpha=\bm\alpha$ and $P\bm Z={\bf 0}$ for any $\bm Z\in\mathcal{Z}$.  Moreover,
\begin{equation}\label{RT2}
    P\mathcal{L}\bm X=\bm 0,\quad \forall\;\bm X\in\mathbb{R}^n.
\end{equation}
Note that $P\bm X\in{\rm span}(\bm \alpha)$ and $({\rm id}-P)\bm X\in\mathcal{Z}$ for any $\bm X\in\mathbb{R}^n$. Furthermore, there is  positive number $\nu>0$ and $M_\nu>0$ \cite[Theorem 1.5.3]{DanHenry} such that 
\begin{equation}\label{RT1}
    \|e^{t\mathcal{L}}({\rm id}-P)\bm X\|_1\le M_{\nu}e^{-\nu t}\|\bm X\|_1, \quad \forall\ \bm X\in\mathbb{R}^n.
\end{equation}
In particular, observing that $e^{t\mathcal{L}}P\bm X=P\bm X$, we have that $e^{t\mathcal{L}}\bm X\to P\bm X$ as $t\to\infty$ for any $\bm X\in\mathbb{R}^n$.
\medskip

The following lemma will be useful in the proofs of our main results.

\begin{lem}\label{lem0} Suppose that  {\bf (A1)} holds. 
    Let $d>0$  and $\bm F :\mathbb{R}_+\to \mathbb{R}^n$ be a continuous map satisfying $\|\bm F(t)\|_1\to 0$ as $t\to\infty$. If  $\bm X(t)$ be a bounded solution of the system
    \begin{equation*}
       \begin{cases}
       \bm X'(t)=d\mathcal{L}\bm X(t) +\bm F(t),\ t>0,\cr 
        \bm X(0)=\bm X^0\in\mathbb{R}^n,
        \end{cases}
    \end{equation*}
    then $\bm X(t)-({\sum_{j\in\Omega} X_j(t)})\bm\alpha\to\bm 0$ as $t\to\infty$. 
    In particular if $\bm F(t)={\bf 0} $ for all $t\ge 0$,  then $\bm X(t)\to ({\sum_{j\in\Omega}X^0_j})\bm\alpha$ as $t\to\infty$.
\end{lem}
\begin{proof}
We write $\bm X(t)=P\bm X(t)+({\rm id}-P)\bm X(t)$ for any $t\ge 0$, where $P$ is as in \eqref{RT1}. Then $P\bm X(t)=g_{X}(t)\bm\alpha$ for some $g_{X}(t)\in\mathbb{R}$,  $t\ge 0$. Next, thanks to \eqref{RT2},
\begin{align*}
\frac{d}{dt}({\rm id}-P)\bm X=& \frac{d\bm X}{dt}-P\frac{d\bm X}{dt}\cr 
=& d\mathcal{L}({\rm id}-P)\bm X+\bm F(t)-dP\mathcal{L}\bm X-P\bm F(t)\cr
=& d\mathcal{L}({\rm id}-P)\bm X +({\rm id}-P)\bm F(t).
\end{align*}
Hence, by the variation of constant formula,
\begin{equation}\label{RT4-2}
    ({\rm id}-P)\bm X(t+\tau)=e^{d\tau \mathcal{L}}({\rm id}-P)\bm X(t)+\int_0^{\tau}e^{d(\tau-s)\mathcal{L}}({\rm id}-P)\bm F(t+s)ds,\quad t\ge 0, \tau>0.
\end{equation}
By \eqref{RT1} and \eqref{RT4-2},
\begin{align}\label{RT5-2}
    \|({\rm id}-P)\bm X(t+\tau)\|_1\le & M_{\nu}e^{-\tau\nu d}\|\bm X(t)\|_1+M_{\nu}\int_0^{\tau}e^{-d\nu(\tau-s)}\|\bm F(t+s)\|_1ds\cr 
    \le & M_{\nu}e^{-d\tau\nu}\sup_{t'\ge 0}\|\bm X(t')\|_1+M_{\nu}\Big[\int_0^{\tau}e^{-d\nu(\tau-s)}ds\Big]\sup_{t'\ge t}\|\bm F(t')\|_1\cr 
    \le & M_{\nu}e^{-d\tau\nu}\sup_{t'\ge 0}\|\bm X(t')\|_1+\frac{M_{\nu}}{d\nu}\sup_{t'\ge t}\|\bm F(t')\|_1.
\end{align}
Since $\|\bm F(t)\|_1\to 0$ as $t\to\infty$, we  conclude from \eqref{RT5-2} that $\|({\rm id}-P)\bm X(t)\|_1\to 0 $ as $t\to\infty$.  As a result, since $\bm X(t)=g_{X}(t)\bm\alpha + ({\rm id}-P)\bm X(t)$, then 
$$
\sum_{j\in\Omega} X_j(t)=g_{X}(t)+\sum_{j\in\Omega}\big<\bm e_j,({\rm id}-P)\bm X(t)\big>,\quad \forall\ t\ge 0,
$$ 
where  $\{\bm e_1,\cdots,\bm e_n\}$ is the canonical basis of $\mathbb{R}^n$. Hence, 
$$
\Big|g_{X}(t)-{\sum_{j\in\Omega} X_j(t)}\Big|\le \|({\rm id}-P)\bm X(t)\|_1\to 0\quad \text{as} \quad t\to\infty.
$$
Therefore,
\begin{align*}
\Big\|\bm X(t)-\bm\alpha{\sum_{j\in\Omega} X_j}\Big\|_1=&\Big\|P\bm X(t)+({\rm id}-P)\bm X(t)-\bm\alpha{\sum_{j\in\Omega} X_j}\Big\|_1\cr 
= & \Big\|({\rm id}-P)\bm X(t)+\Big(g_{X}(t)-{\sum_{j\in\Omega} X_j}\Big)\bm\alpha\Big\|_1\cr 
\le & \|({\rm id}-P)\bm X(t)\|_1+\Big|g_{X}(t)-({\sum_{j\in\Omega} X_j})\Big|\to 0\quad \text{as}\quad t\to\infty.
\end{align*}
This completes the proof of the first assertion of the lemma. If in addition $\bm F(t)={\bf 0}$ for any $t\ge 0$, then 
$$
\frac{d}{dt}\sum_{j\in\Omega} X_j(t)=0,\quad t>0. 
$$
Thus $\sum_{j\in\Omega} X_i(t)=\sum_{j=1} X^0_j$ for all $t\ge 0$, which completes the proof of the lemma. 
\end{proof}

 We also recall the following Harnack's inequality type result from \cite{DBS2023}.

\begin{lem}\cite[Lemma 3.1]{DBS2023}\label{Harnck-lemma}  Suppose that  {\bf (A1)} holds. Let $d>0$ and $ \bm M\in C(\mathbb{R}_+, \mathbb{R}^n)$ such that 
\begin{equation*}
    \sup_{t\ge 0}\|\bm M(t)\|_{\infty}\le m_{\infty}<\infty.
\end{equation*}
Then there is a positive number $c_{d,m_{\infty}}$ such that any nonnegative solution $\bm U(t)$ of 
\begin{equation}\label{Harnack-eq1}
    {\bm U'}=d\mathcal{L}\bm U +\bm M(t)\circ \bm U, \ t>0
\end{equation}
satisfies
\begin{equation}\label{Harnack-eq2}
    \|\bm U(t)\|_{\infty}\le c_{d,m_{\infty}}\bm U_{m}(t),\quad \forall\ t\ge 1.
\end{equation}
    
\end{lem}
    
We will also need the following result.

\begin{lem}\label{lem4} Let $F:[0,\infty)\to [0,\infty)$ be a  H\"older continuous function satisfying 
$$ 
\sup_{t\ge a}F(t)e^{\int_a^tF(s)ds}<\infty 
$$
for some nonnegative number $a\ge 0$. Then $F(t)\to 0$ as $t\to\infty$.
\end{lem}
\begin{proof} Let $\bar{F}:=\limsup_{t\to \infty}F(t)$. It suffices to show that $\bar{F}=0$. Suppose to the contrary that $\bar F\neq 0$. Then there is a sequence $\{t_n\}_{n\ge 1}$ converging to infinity such that $\lim_{n\to\infty}F(t_n)=\bar{F}\in(0,\infty]$. Define $K:=\sup_{t\ge a}F(t)e^{\int_a^tF(s)ds}$. Then 
\begin{equation}\label{RT11}
\int_a^{t_n}F(s)ds\le \ln\Big(\frac{K}{F(t_n)}\Big),\quad \quad \forall\ n\gg 1.
\end{equation}
Since $F(t_n)\le K$ for $n\gg 1$,  $\bar{F}<\infty$. It then follows from \eqref{RT11} that 
$\int_a^{\infty}F(s)ds\le\ln( K/{\bar{F}})<\infty$. Since  $F(t)\ge 0$ for all $t\ge 0$ and $F$ is H\"older continuous, we have $F(t)\to 0$ as $t\to\infty$. This is a contradiction. Therefore  $\bar{F}=0$, which yields the desired result.
\end{proof}

Our approach to study the existence and uniqueness/non-uniqueness of EE of \eqref{model-mass-action} when the dispersal rates are positive is to first transform the equilibrium problem into an equivalent problem. To carry out this approach we will need the following two lemmas.

\begin{lem}\label{lem2} Suppose that  {\bf (A1)}-{\bf (A3)}  holds. Let $\bm\alpha$ be defined as above. Consider the one-parameter family of algebraic equations:
\begin{equation}\label{Eq1}
    d_I\mathcal{L}\bm U +(l\bm\beta\circ(N\bm\alpha-d_I\bm U)-\bm\gamma)\circ\bm U=\bm 0, \quad l>0.
\end{equation}
The following conclusions hold.
\begin{itemize}
    \item[\rm (i)] System \eqref{Eq1} has a (unique) positive solution, $\bm U^l$,  if and only if $l>1/{\mathcal{R}_0}$. Moreover the positive solution $\bm U^l$, if exists, satisfies that
    \begin{equation}\label{lem2-eq1}
        {\bf 0}<d_I\bm U^l<N\bm\alpha,\quad \lim_{l\to {1}/{\mathcal{R}_0}}\|\bm U^l\|_1=0,\quad \lim_{l\to\infty}\|d_I\bm U^l-N\bm\alpha\|_1=0,\quad \text{and}\quad \lim_{l\to\infty}\|l(N\bm\alpha-d_I\bm U^l)-\bm r\|_1=0.
    \end{equation}
    \item[\rm (ii)] The mapping $\big({1}/{\mathcal{R}_0},\infty\big)\ni l\mapsto \bm U^l$ is smooth and strictly increasing.

    \item[\rm (iii)] The function $\mathcal{N}:(1/\mathcal{R}_0, \infty)\to \mathbb{R}_+$ defined by 
    \begin{equation}\label{N-d_I-def}
        \mathcal{N}(l)=l\sum_{j\in\Omega}(N\alpha_j-d_I U^l_j),\quad \forall\ l> \frac{1}{\mathcal{R}_0}
    \end{equation}
    is continuously differentiable and satisfies  \begin{equation}\label{lem2-eq2}
        \lim_{l\to {1}/{\mathcal{R}_0}}\mathcal{N}(l)=\frac{N}{\mathcal{R}_0}\quad \text{and}\quad \lim_{l\to\infty}\mathcal{N}(l)=\sum_{j\in\Omega} r_j.
    \end{equation} 
\end{itemize}

\end{lem}
\begin{proof} {\rm (i)} Since  \eqref{Eq1} is of logistic type, it has a (unique) positive solution if and only the trivial solution is unstable, i.e., $\sigma_*(d_I\mathcal{L}+{\rm diag}(lN\bm\alpha\circ\bm\beta-\bm\gamma))>0$\cite{cosner1996variability,li2010global,Lu1993}. Since $l\mathcal{R}_0-1$ has the same sign as $\sigma_*(d_I\mathcal{L}+{\rm diag}(lN\bm\alpha\circ\bm\beta-\bm\gamma))$, the positive solution exists if and only if $l>1/\mathcal{R}_0$.

Suppose that $l>1/\mathcal{R}_0$ such that the positive solution $\bm U^l$ exists. 
Since $\bar{\bm U}:=\frac{N}{d_I}\bm\alpha$ is a strict super solution of \eqref{Eq1}, we conclude that $d_I\bm U^l<N\bm\alpha$. Moreover since \eqref{Eq1} has no positive solution when $l={1}/{\mathcal{R}_0}$ and $\|\bm U^{l}\|_{\infty}<{N\|\bm\alpha\|_\infty}/{d_I}$ for any $l>{1}/{\mathcal{R}_0}$,  we must have that $\|\bm U^{l}\|_1\to 0$ as $l\to {1}/{\mathcal{R}_0}$ from the right. 

Let $\bm Z^l:=l(N\bm\alpha-d_I\bm U^l)$. By \eqref{Eq1},  $\bm Z^l$ satisfies
\begin{equation}\label{Z-l=eq}
  \frac{1}{l}\mathcal{L}\bm Z^{l} +\bm\beta\circ(\bm r-\bm Z^l)\circ\bm U^l=\bm 0. 
\end{equation}
It then follows from the comparison principle that 
\begin{equation}\label{Z-l-bounds}
\frac{\bm r_{m}}{\bm\alpha_M}{\bm \alpha}\le\bm Z^l\le\frac{\bm r_{M}}{\bm\alpha_m}{\bm \alpha},\quad  \forall\  l>{1}/{\mathcal{R}_0}.
\end{equation} 
Hence,
$$
\|d_I\bm U^l-N\bm\alpha\|_1=\frac{1}{l}\|\bm Z^l\|_1\le \frac{\bm r_M}{l\bm\alpha_m}\to 0, \quad \text{as}\ l\to\infty.
$$
As a result,  
$$
\Big\|\frac{1}{l}\mathcal{L}\bm Z^l\Big\|_1\le \frac{1}{l}\|\mathcal{L}\|_1\|\bm Z^l\|_1\to0, \quad \text{as}\quad l\to\infty.
$$
Taking $l\to \infty$ in \eqref{Z-l=eq} yields that $\bm Z^l\to\bm r$. This completes the proof of \eqref{lem2-eq1}.

{\rm (ii)} It is clear from \eqref{Eq1} that $\bm U^{l+h}$ is a strict supersolution of the equation satisfied by $\bm U^l$ for any $l>{1}/{\mathcal{R}_0}$ and $h>0$. Therefore, $\bm U^{l}\ll \bm U^{l+h}$ for any $l>{1}/{\mathcal{R}_0}$ and $h>0$. Next, observe that the expression at the left hand side of \eqref{Eq1} is smooth in $l$ and $\bm U$. Moreover if \eqref{Eq1} is linearized at $\bm U^{l}$, we obtain the eigenvalue problem
$$
\lambda \bm \varphi=d_{I}\mathcal{L}\bm \varphi+(l\beta\circ(N\bm\alpha-2d_I\bm U^l)-\bm \gamma)\circ \bm \varphi.
$$
By {\bf (A1)} and the Perron-Frobenius theorem $\sigma_*(d_I\mathcal{L}+{\rm diag}(l\bm\beta\circ(N\bm\alpha-2d_I\bm U^l)-\bm \gamma))$ is the principal eigenvalue of the matrix $d_I\mathcal{L}+{\rm diag}(l\bm\beta\circ(N\bm\alpha-2d_I\bm U^l)-\bm \gamma)$. Furthermore,
\begin{align*}
\sigma_*(d_I\mathcal{L}+{\rm diag}(l\bm\beta\circ(N\bm\alpha-2d_I\bm U^l)-\bm\gamma))=&\sigma_*(d_I\mathcal{L}+{\rm diag}(l\bm\beta\circ(N\bm\alpha-d_I\bm U^l)-\bm\gamma)-ld_I{\rm diag}(\bm\beta\circ\bm U^l))\cr 
\le& \sigma_*(d_I\mathcal{L}+{\rm diag}(l\bm\beta\circ(N\bm\alpha-d_I\bm U^l)-\bm\gamma))-ld_I\bm\beta_m\bm U^l_m.
\end{align*}
Since $\bm U^l$ satisfies \eqref{Eq1}, $\sigma_*(d_I\mathcal{L}+{\rm diag}(l\bm\beta\circ(N\bm\alpha-d_I\bm U^l)-\bm\gamma))=0$. It follows that $\sigma_*(d_I\mathcal{L}+{\rm diag}(l\bm\beta\circ(N\bm\alpha-2d_I\bm U^l)-\bm\gamma))<0$. Therefore, $d_I\mathcal{L}+{\rm diag}(l\bm\beta\circ(N\bm\alpha-2d_I\bm U^l)-\bm\gamma)$ is invertible. By  the implicit function theorem, $\bm U^l$ depends smoothly on $l$.

{\rm (iii)} The regularity of the function $\mathcal{N}$ in $l$ follows from {\rm (ii)}. The asymptotic limits in \eqref{lem2-eq2} follow from \eqref{lem2-eq1} and the fact that $\sum_{j\in\Omega}\alpha_j=1$.
\end{proof}

\begin{lem}\label{lem3} Suppose that  {\bf (A1)}-{\bf (A3)}  holds.
\begin{itemize}
    \item[\rm (i)] Let $(\bm S,\bm I)$ be an EE solution of \eqref{model-mass-action} and define 
    \begin{equation}\label{kappa-def}
        \bm\kappa =d_S\bm S+d_I\bm I.
    \end{equation} 
    Then there is a positive number $\kappa^*$ such that $\bm \kappa=\kappa^*N\bm\alpha$, where $\kappa^*N=d_SN+(d_I-d_S)\sum_{i\in\Omega}I_i$ and $\bm\alpha$ is given by \eqref{alpha-eq}.  Furthermore, setting 
    \begin{equation}\label{tilde-s-i-def}
        \tilde{\bm S}=\frac{\bm S}{\kappa^*}\quad \text{and}\quad \tilde{\bm I}=\frac{\bm I}{\kappa^*},
    \end{equation}
    then $\tilde{\bm I}$ is the positive solution of \eqref{Eq1} with $l={\kappa^*}/{d_S}$. Moreover, ${\kappa^*}/{d_S}>{1}/{\mathcal{R}_0}$.
    
    \item[\rm (ii)] If $l>{1}/{\mathcal{R}_0}$ and  
    \begin{equation}\label{N-equation}
        N=l\left[\sum_{j\in\Omega}(N\alpha_j-d_I U_j^l)+d_S\sum_{j\in\Omega} U_j^l\right],
    \end{equation}
    then $(\bm S,\bm I):=(l(N\bm\alpha -d_I\bm U^l),d_Sl\bm U^l)$ is an EE solution of \eqref{model-mass-action}, where $\bm U^l$ is the solution of \eqref{Eq1}.
\end{itemize}
\end{lem}
\begin{proof} {\rm (i)}  Let $(\bm S,\bm I)$ be an EE solution of \eqref{model-mass-action} and $\bm \kappa$ be  defined by \eqref{kappa-def}. Then $\mathcal{L}\bm\kappa=0$. By the Perron-Frobenius theorem, there is a real number $\kappa^*$ such that $\bm\kappa=\kappa^*N\bm\alpha$. Since  $\bm S, \bm I\gg \bm 0$,  we have $\bm\kappa\gg\bm 0$ and $\kappa^*>0$. By $\sum_{i\in\Omega}(S_i+I_i)=N$, it is easy to see  $\kappa^*N=d_SN+(d_I-d_S)\sum_{i\in\Omega}I_i$. 

Let $(\tilde{\bm S},\tilde{\bm I})$ be defined as in \eqref{tilde-s-i-def}.  By \eqref{kappa-def},  $\bm S=\frac{\kappa^*}{d_S}(N\bm\alpha-d_I\tilde{\bm I})$. Taking $l={\kappa^*}/{d_S}$ and by the second equation of \eqref{model-mass-action},  $\tilde{\bm I}$ solves \eqref{Eq1}. Since  \eqref{Eq1} has a positive solution  $\tilde{\bm I}$, Lemma \ref{lem2}-{\rm (i)} implies that ${\kappa^*}/{d_S}>{1}/{\mathcal{R}_0}$.

{\rm (ii)} Suppose that $l>{1}/{\mathcal{R}_0}$. By Lemma \ref{lem2}, \eqref{Eq1} has a positive solution $\bm U^l$. 
 If $\bm U^l$ satisfies \eqref{N-equation},  it is straightforward to verify that $(\bm S,\bm I)=(l(N\bm\alpha-d_I\bm U^l),ld_S\bm U^l)$ is an EE solution of \eqref{model-mass-action}.
\end{proof}

\section{Global stability and EE solutions}
For simple epidemic differential equation models, one usually expects that the global dynamics is determined by $\mathcal{R}_0$: if $\mathcal{R}_0< 1$, the DFE is globally stable; if $\mathcal{R}_0>1$, the model has a globally stable (and unique) EE. This is true if model  \eqref{model-mass-action} has only one patch. However when $n\ge 2$, there are certain parameter ranges that do not fit the expectations.

\subsection{Case $\mathcal{R}_0\le 1$}

Recall that the DFE ${\bf E}^0:=({N}\bm\alpha, {\bf 0})$ is unique. Firstly, we find parameter ranges such that the DFE is globally stable when $\mathcal{R}_0\le 1$. 

\begin{tm}[Global stability of DFE when $\mathcal{R}_0\le 1$]\label{T1-mass-action-DEF-stability} Suppose that {\bf (A1)-(A3)} holds. The DFE is linearly stable if $\mathcal{R}_0<1$ and unstable if $\mathcal{R}_0>1$. Furthermore, the DFE is globally asymptotically stable if one of the following conditions hold:
\begin{itemize}
    \item[\rm (i)]  $N\rho({\rm diag}(\bm \beta)\mathcal{V}^{-1})\le 1$;
    
    \item[\rm (ii)]  $\mathcal{R}_0<1$ and  $d_S>d_{\rm low}(d_I,N)$, for some positive number $d_{\rm low}(d_I,N)$ depending on $d_I$ and $N$;
    
    \item[\rm (iii)]  $\mathcal{R}_0\le 1$ and $d_S=d_I$;
    
    \item[\rm (iv)] $\mathcal{R}_0\le 1$ and $\bm\gamma=m\bm\beta\circ\bm\alpha$ for some positive number $m>0$.
\end{itemize}
\end{tm}
\begin{proof} The linear stability/instability of the DFE in terms of $\mathcal{R}_0$ is standard \cite[Theorem 2]{Driessche}. So it remains to prove the global stability of the DFE if one of {\rm (i)}-{\rm (iv)} holds.

{\rm (i)} Suppose that $N\rho({\rm diag}(\bm\beta)\mathcal{V}^{-1})\le 1$. Let $(\bm S(t), \bm I(t))$ be a solution of \eqref{model-mass-action}. Let $\tilde{\bm E}\in\mathbb{R}^n$ be the positive eigenvector associated with $\sigma_*(d_I\mathcal\mathcal{L} +{\rm diag}(N\bm\beta -\bm\gamma) )$ with $\tilde{\bm E}_m=1$. {Define  
$$
\bar{\bm I}(t):=e^{t\sigma_*(d_I\mathcal\mathcal{L}+{\rm diag}(N \bm \beta-\bm \gamma))}\tilde{\bm E},\quad\forall\ t\ge0.
$$
It holds that
\begin{equation}\label{KPL1}
    \frac{d \bar{\bm I}}{dt}=d_I\mathcal{L}\bar{\bm I}+(N\bm \beta-\gamma)\circ \bar{\bm I},\quad \forall\ t\ge 0.
\end{equation}
Let 
$$
c(t):=\inf\{\tilde c>0:\ \bm I(t)\le \tilde c\tilde{\bm E}\}, \quad t\ge 0.
$$
Then, 
\begin{equation}\label{DE1}
    \bm I(t)\le c(t)\tilde{\bm E},\quad \forall\ t\ge 0.
\end{equation}
By the definition of $c(t)$, there is $j_t\in\Omega$ such that  $I_{j_t}(t)=c(t)\tilde{ E}_{j_t}$ for all $t\ge 0$.
It follows that 
$$
\sum_{j\in\Omega} I_j(t)\ge  I_{j_t}(t)= c(t)\tilde{ E}_{j_t}\ge c(t)\min_{j\in\Omega}\tilde{ E}_j=c(t)\tilde{\bm E}_m=c(t),\quad \forall\ t\ge 0.
$$
As a result,
\begin{equation*}
    \sum_{j\in\Omega} S_j(t)=N-\sum_{j\in\Omega} I_j(t)\le N-c(t),\quad \forall\ t\ge 0.
\end{equation*}
This in turn implies that
\begin{equation}\label{KPL2}
   \bm I'\le d_I\mathcal{L}\bm I+((N-c(t))\bm \beta-\bm \gamma)\circ \bm I\le d_I\mathcal{L}\bm I+(N\bm\beta-\bm \gamma)\circ \bm I-\bm\beta_mc(t)\bm I,\quad \forall\ t>0.
\end{equation}
Hence, setting 
$$ 
\underline{\bm I}(t+t_0;t_0):=e^{\bm\beta_m\int_{t_0}^{t_0+t}c(s)ds}\bm I(t+t_0),\quad\ \forall\, t,t_0\ge0,
$$
we obtain from \eqref{KPL2} that 
\begin{eqnarray}\label{KPL3}
    \frac{d\underline{\bm I}(t+t_0;t_0)}{dt}\le d_I\mathcal{L}\underline{\bm I}(t+t_0;t_0)+(N\bm\beta-\gamma)\circ\underline{\bm I}(t+t_0;t_0),\quad \forall \, t>0,t_0\ge0.
\end{eqnarray}
Observing that 
$$
\underline{\bm I}(t_0;t_0)=\bm I(t_0)\le c(t_0)\tilde{\bf E}=c(t_0)\bar{\bm I}(0),
$$
it follows from \eqref{KPL1}, \eqref{KPL3}, and  the comparison principal for monotone dynamical systems that 
$$
\underline{\bm I}(t+t_0;t_0)\le c(t_0)\bar{\bm I}(t),\quad \forall \,t>0,t_0\ge 0.
$$
Equivalently,
\begin{equation}\label{DE2}
e^{\bm\beta_m\int_{t_0}^{t_0+t}c(s)ds}\bm I(t+t_0)\le c(t_0)e^{t\sigma_*(d_I\mathcal\mathcal{L}+{\rm diag}(N \bm \beta-\bm \gamma))}\tilde{\bm E},\quad \forall\ t_0,t\ge 0.
\end{equation}
}
Recall that $ \sigma_*(d_I\mathcal\mathcal{L}+{\rm diag}(N \bm \beta-\bm \gamma))\le 0$ since $N\rho({\rm diag}(\bm \beta)\mathcal{V}^{-1})\le 1$. Then by the definition of $c(t)$, we get from \eqref{DE2} that 
$$
c(t_0+t)\le c(t_0)e^{-\bm\beta_m\int_{t_0}^{t_0+t}c(s)ds}, \quad \forall\ t,t_0\ge 0.
$$
This implies that $c(t)$ is nonincreasing in $t$ and 
$$
c(t)\le c(0)e^{-\bm\beta_m\int_{0}^tc(s)ds},\quad \forall\ t\ge 0.
$$
Hence by Lemma \ref{lem4},  $c(t)\to0$ as $t\to\infty$. We then conclude from \eqref{DE1} that $\|\bm  I(t)\|_\infty\to 0$ as $t\to\infty$.

It remains to show that $\bm S(t)\to {N}\bm\alpha$ as $t\to\infty$. By Lemma \ref{lem0}, we know that
$\bm  S(t)-\bm\alpha{\sum_{j\in\Omega} S_j(t)}\to 0$ as $t\to\infty$. Since $\sum_{i}( S_i(t)+ I_i(t))=N$ for all $t\ge 0$, we have $\sum_{i\in\Omega} S_i(t)\to N$ as $t\to\infty$. It follows that $\bm S(t)\to {N}\bm\alpha$ as $t\to\infty$.

{\rm (ii)} Fix $d_I, N>0$ such that $\mathcal{R}_0<1$.  Taking $t=0$ in \eqref{RT5-2} with $\bm X$ replaced by $\bm S$, we obtain that 
\begin{equation*}
    \sup_{\tau\ge 1}\|({\rm id}-P)\bm S(\tau)\|_1\le NM_{\nu}(1+\|\bm\gamma\|_{\infty}+N\|\bm\beta\|_{\infty})\Big(e^{-d_S\nu}+\frac{1}{d_S\nu}\Big):=\tilde{M}_{\nu}\Big(e^{-d_S\nu}+\frac{1}{d_S\nu}\Big),
\end{equation*}
where we have used the fact that $\|\bm S(t)+\bm I(t)\|_1=N$ for all $t\ge 0$. Next, we observe that  
\begin{align*}
N=&\sum_{j\in\Omega}\big<\bm e_j,P\bm S(t)\big>+\sum_{j\in\Omega}\big<\bm e_j,({\rm id}-P)\bm S(t)\big>+\sum_{j\in\Omega} I_{j}\cr 
=&g_S(t)+\sum_{j\in\Omega}\big<\bm e_j,({\rm id}-P)\bm S(t)\big>+\sum_{j\in\Omega} I_{j},
\end{align*}
where   $\{\bm e_1,\cdots,\bm e_n\}$ is the canonical basis of $\mathbb{R}^n$ and $g_S(t)\bm\alpha=P\bm S(t)$. 
Then
\begin{align*}
    \bm I'=&d_I\mathcal{L}\bm I+\Big(\bm \beta \circ P\bm S(t)+\bm \beta\circ ({\rm id}-P)\bm S-\bm \gamma )\circ \bm I\cr
    =& d_I\mathcal{L}\bm I+\big(g_S(t)\bm \beta \circ \bm \alpha+\bm \beta\circ ({\rm id}-P)\bm S-\bm \gamma \big)\circ\bm  I\cr 
    =& d_I\mathcal{L}\bm I+\Big(N\bm \beta \circ \bm \alpha-\bm \gamma+\bm \beta\circ\Big( ({\rm id}-P)\bm S-\bm\alpha\sum_{j\in\Omega}I_j -\bm\alpha\sum_{j\in\Omega}\big<\bm e_j,({\rm id}-P)\bm S\big>\Big) \Big)\circ \bm I\cr 
    \le & d_I\mathcal{L}\bm I+\Big(N\bm \beta \circ \bm \alpha-\bm \gamma+2\|({\rm id}-P)\bm S\|_1\bm \beta \Big)\circ\bm  I\cr
    \le & d_I\mathcal{L}\bm I+\Big(N\bm \beta \circ \bm \alpha-\bm\gamma+2\|\bm \beta\|_{\infty}\tilde{M}_{\nu}\Big(e^{-d_S\nu}+\frac{1}{d_S\nu}\Big){\bf 1} \Big)\circ \bm I,\quad \forall\ t\ge 1.
\end{align*}
Let $\hat{\bm E}\in \mathbb{R}^n$ be the positive eigenvector of $\sigma_*(d_I\mathcal{L}+{\rm diag}(N\bm \beta\circ\bm\alpha-\bm \gamma))$ satisfying $\hat{\bm E}_m=1$. By $\bm I(1)\le N\bm 1\le  N\hat{\bm E}$ and  the comparison principle for cooperative systems, we get  
\begin{align}\label{RT6}
\bm I(t+1)\le Ne^{2\|\bm \beta\|_{\infty}\tilde{M}_{\nu}\Big(e^{-d_S\nu}+\frac{1}{d_S\nu}\Big)t}e^{\sigma_*t}\hat{\bm E},\quad t\ge 0,
\end{align}
where $\sigma_*:=\sigma_*(d_I\mathcal{L}+{\rm diag}(N\bm \beta\circ\bm\alpha-\bm \gamma))$. Since $\mathcal{R}_0<1$ and $\sigma_*$
has the same sign as $\mathcal{R}_0-1$,  we have $\sigma_*<0$. So there exists  $d_{\rm low}(d_I,N)>0$ such that 
$$
\sigma_*+2\|\bm \beta\|_{\infty}\tilde{M}_{\nu}\Big(e^{-d_S\nu}+\frac{1}{d_S\nu}\Big)<0, \quad d_S>d_{\rm low}(d_I,N).
$$
It then follows from \eqref{RT6} that $\|\bm I(t)\|_1\to 0$ as $t\to\infty$ for any $d_S\ge d_S^*$. We can  proceed as in the proof of {\rm (i)} to conclude that $\bm S(t)\to N\bm \alpha$ as $t\to\infty$.

{\rm (iii)} Suppose that $d:=d_S=d_I$ and $\mathcal{R}_0\le 1$.  Let $\bm Z(t)=\bm S(t)+\bm I(t)$. Then $
\bm Z'=d\mathcal{L}\bm Z, \quad t\ge 0$.
By $\sum_{j\in\Omega} Z_j(t)=N$ for all $t\ge 0$ and  Lemma \ref{lem0}, we have
\begin{equation}\label{RT7}
    \bm Z(t)=\bm S(t)+\bm I(t)\to N\bm\alpha \quad \text{as}\quad t\to\infty.
\end{equation}
Let $0<\varepsilon\ll 1$ be given. Then there is $t_{\varepsilon}>0$ such that 
\begin{equation}\label{RT7-2}
 (1-\varepsilon)N\bm \alpha-\bm I(t)\le \bm S(t)\le (1+\varepsilon)N\bm \alpha-\bm I(t), \quad \forall\ t\ge t_{\varepsilon}.
 \end{equation}
This  implies that 
\begin{equation}\label{RT8}
    \bm I'\le d_I\mathcal{L}\bm I+\big( \bm\beta\circ\big((1+\varepsilon)N\bm\alpha -\bm I(t) \big)-\bm \gamma\big)\circ\bm  I,\quad \forall\ t\ge t_{\varepsilon}.
\end{equation}
Let $\bm I_{\varepsilon}$ be the unique nonnegative stable equilibrium of the logistic type system
\begin{equation}\label{RT9}
    \bm I'=d_I\mathcal{L}\bm I+\big( \bm \beta\circ\big((1+\varepsilon)N\bm\alpha -\bm I(t) \big)-\bm \gamma\big)\circ \bm I.
\end{equation}
Then $\bm I_{\varepsilon}$ is nonincreasing in $\varepsilon>0$, and $\bm I_\epsilon\to\bm I_0$ as $\varepsilon\to 0$, where $\bm I_{0}$ solves \eqref{RT9} with $\varepsilon=0$. By \eqref{RT8} and the comparison principle, we have that 
$$
\limsup_{t\to\infty}\|\bm I(t)\|_1\le \|\bm  I_{\varepsilon}\|_1,\quad \forall\ \varepsilon>0.
$$
Taking $\varepsilon\to0$  gives \begin{equation}\label{RT10}
    \limsup_{t\to\infty}\|\bm I(t)\|_1\le \|\bm I_0\|_1.
\end{equation}
Since $\mathcal{R}_0\le 1$, we have $\sigma_*\le 0$ and $\bm I_0={\bf 0}$. So \eqref{RT10} yields that $\|\bm I(t)\|_1\to 0$ as $t\to\infty$. This together with \eqref{RT7} yields that $\bm S(t)\to N\bm\alpha$ as $t\to\infty$.

{\rm (iv)} Suppose that $\mathcal{R}_0\le 1$ and $\bm \gamma=m\bm \beta\circ \bm\alpha$ for some $m>0$. In this case, we have that $\mathcal{R}_0={N}/{m}$, which yield that $m\ge N$. We can rewrite the equation of $\bm S$ as
\begin{equation}\label{RT12}
    \frac{d\bm S}{dt}=d_S\mathcal{L}\bm S+\bm\beta\circ(m \bm\alpha -\bm S)\circ \bm I,\quad t>0.
\end{equation}
Since $\sup_{t\ge 0}\|(\bm\beta\circ \bm S(t)-\bm \gamma)\circ \bm I(t)\|_{\infty}\le N(\|\bm\beta\|_{\infty}+N\|\bm\gamma\|_{\infty})<\infty$, it follows from Lemma \ref{Harnck-lemma} that there is  $c_0>0$ such that 
\begin{equation}\label{RT13}
    \|\bm  I(t)\|_{\infty}\le c_0 \bm I_m(t),\quad \forall\ t\ge 1.
\end{equation}
Define $ \underline{\bm S}(t)=m\bm\alpha - me^{-{\bm\beta}_m\int_{1}^t \bm I_m(s)ds}\bm\alpha$, $t\ge 1$. Then $ \underline{\bm S}(1)={\bf 0}<\bm S(1)$ and 
\begin{align}
    \frac{d \underline{\bm S}}{dt}-d\mathcal{L}\underline{\bm S}-\bm \beta\circ(m\bm\alpha-\underline{\bm S})\circ \bm I =& m{\bm\beta}_m \bm I_m(t)e^{-{\bm\beta}_m\int_{1}^t \bm I_m(s)ds}\bm\alpha- me^{-{\bm\beta}_m\int_{1}^t \bm I_m(s)ds}\bm \beta\circ\bm\alpha\circ \bm I(t)\cr
    =& me^{-{\bm \beta}_m\int_{1}^t \bm I_m(s)ds}\Big({\bm \beta}_m \bm I_m(t){\bf 1}-\bm \beta\circ \bm I(t)\Big)\circ\bm\alpha\le {\bf 0}.
\end{align}
Hence by the comparison principle for monotone systems, we have  
\begin{equation}\label{Sl0}
\underline{\bm S}(t)\le \bm S(t), \quad \forall\ t\ge 1,
\end{equation}
which implies that 
\begin{equation}\label{RT13-2}
   m(1-e^{-{\bm\beta}_m\int_{1}^t \bm I_m(s)ds})=\sum_{j\in\Omega}\underline{S}_j(t)\le \sum_{j\in\Omega} S_j(t)=N-\sum_{j\in\Omega} I_j(t),\quad \forall\ t\ge 1.
\end{equation}
Define $T(t):=\sum_{j\in\Omega} I_j(t)$ for $t\ge 0$.  By \eqref{RT13}, \eqref{RT13-2}, and $N\le m$, we have
\begin{align*}
    T(t)\le &  me^{-{\bm\beta}_m\int_{1}^t \bm I_m(s)ds}+N-m\cr 
    \le & me^{-{\bm\beta}_m
    \int_{1}^t \bm I_m(s)ds}\cr
    \le &  me^{-\frac{{\bm \beta}_m}{nc_0}\int_1^{t}T(s)ds}. 
\end{align*}
Hence,
\begin{equation}\label{RT14}
T(t) e^{\frac{{\bm \beta}_m}{nc_0}\int_1^tT(s)ds}\le m,\quad \forall\ t\ge 1. 
\end{equation}
Observing that 
$$ 
\sup_{t\ge 1}|T'(t)|=\sup_{t\ge 1}|\sum_{j\in\Omega}(\beta_j S_j-\gamma_j) I_j(t)|\le (\|\bm \beta\|_{\infty}N+\|\bm \gamma\|_{\infty})N<\infty,
$$
$T:\mathbb{R}_+\to\mathbb{R}_+$ is  H\"older continuous. 
It follows from \eqref{RT14} and Lemma \ref{lem4} that $T(t)=\|\bm I(t)\|_1\to 0$ as $t\to\infty$. We can now proceed as in the proof of {\rm (i)} to deduce that $\bm S(t)\to N\bm\alpha$ as $t\to\infty$.   
\end{proof}

\begin{rk}
Theorem \ref{T1-mass-action-DEF-stability} provides some sufficient conditions  under which the unique DFE is globally asymptotically stable when $\mathcal{R}_0\le 1$. These results extend some of the known results when $L$ is symmetric. In particular, if $L$ is symmetric, Theorem \ref{T1-mass-action-DEF-stability}-{\rm (i-1)} $\&$ {\rm (i-2)} recover \cite[Theorem 2.1-{\rm (ii-1) $\&$ {\rm (ii-2)}} ]{DBS2023}, and Theorem \ref{T1-mass-action-DEF-stability}-{\rm (i-3)} $\&$ {\rm (i-4)} recover \cite[Theorem 2 {\rm (i)} and Theorem 3-{\rm (i)}]{li2023sis}. We point out that many techniques  for the  symmetric case don't work for the general case.  So we have developed new tools here, which might be helpful to handle similar problems in the future.
\end{rk}

Under any of the assumptions (i)-(iv) in Theorem \ref{T1-mass-action-DEF-stability}, the DFE is globally stable and hence EE does not exist. In the following result, we identify  more parameter ranges such that  EE does not exist. We conjecture that  the DFE is globally stable in these parameter ranges.  

\begin{prop}[Non-existence of EE when $\mathcal{R}_0\le 1$]\label{T1-mass-action-EE-Non-Existence} Suppose that {\bf (A1)-(A3)} holds and $\mathcal{R}_0\le 1$. Then \eqref{model-mass-action} has no EE solution under one of the following assumptions: 
\begin{itemize}
    \item[\rm (i)] $d_S\ge d_I\mathcal{R}_0$;
    \item[\rm (ii)] $N\le ({\bf r}/{\bm \alpha})_m$.
\end{itemize}    
\end{prop}
\begin{proof} ]
{\rm (i)} Suppose $d_S\ge d_I\mathcal{R}_0$.  Assume to the contrary that there is an EE solution $(\bm S,\bm I)$ of \eqref{model-mass-action}. Then, by Lemma \ref{kappa-def}-{\rm (i)} that $\tilde{\bm I}=\bm U^{l}$ for some $l=\frac{\kappa^*}{d_S}>\frac{1}{\mathcal{R}_0}$, where $\tilde{\bm I}$ is defined in \eqref{tilde-s-i-def} and $\kappa^*$ is the positive number such that $\kappa^*N\bm\alpha=d_S\bm S+d_I\bm I $. It then follows from \eqref{N-equation} and the fact that $\bm U^{l}<\frac{N}{d_I}\bm \alpha$ (see \eqref{lem2-eq1}) that 
\begin{align}\label{RT40}
    N=\sum_{j\in\Omega}( S_j^l+ I_j^l)=&l\Big[\sum_{j\in\Omega}( N\alpha_j-d_I U^l_j)+d_S\sum_{j\in\Omega} U^l_j\Big] \cr
    =&l\Big[N+(d_S-d_I)\sum_{j\in\Omega} U^l_j\Big]\cr 
    \ge & l\Big[N-(d_S-d_I)_{-}\sum_{j\in\Omega} U^l_j \Big]\cr 
    \ge& l\Big[N-(d_S-d_I)_{-}\sum_{j\in\Omega}\frac{ N\alpha_j}{d_I} \Big]\cr
    =&l\Big(1-\Big(\frac{d_S}{d_I}-1\Big)_{-}\Big)N 
    =\min\Big\{1,\frac{d_S}{d_I}\Big\}lN 
    >\min\Big\{1,\frac{d_S}{d_I}\Big\}\frac{N}{\mathcal{R}_0}.
\end{align}
Hence $\mathcal{R}_0>\min\Big\{1,\frac{d_S}{d_I}\Big\}$, which is impossible since $\mathcal{R}_0\le 1 $ and $d_I\mathcal{R}_0\le d_S$. Therefore, \eqref{model-mass-action} has no EE solution.

{\rm (ii)} Suppose $N\le ({\bf r}/{\bm \alpha})_m$. Assume to the contrary that \eqref{model-mass-action} has an EE for some $d_S>0$. Then by the first line in \eqref{RT40}, there is some $l>1/{\mathcal{R}_0}$ such that 
$$
N=\sum_{j\in\Omega}l(\alpha_j-d_I U_j^l)+ld_S\sum_{j\in\Omega}U^l_j>\sum_{j\in\Omega}l(N \alpha_j-d_I U_j^l)=\sum_{j\in\Omega} Z_j^l,
$$
where $\bm Z^l=l(N\bm \alpha-d_I\bm U^l)$ is as in \eqref{Z-l=eq}.  Since $\underline{\bm Z}:=({\bm r}/{\bm \alpha})_m\bm\alpha$ is a subsolution of \eqref{Z-l=eq}, then 
$$
N>\sum_{j\in\Omega}Z_j^l\ge ({\bf r}/{\bm \alpha})_m\sum_{j\in\Omega}\alpha_j=({\bf r}/{\bm \alpha})_m,
$$
which is a contradiction. Hence \eqref{model-mass-action} has no EE solution.    
\end{proof}

In the following result, we find a parameter range such that the model has multiple EE solutions (this also implies that the DFE is not globally stable) when $\mathcal{R}_0<1$, which is a case that does not fit the usual expectation.
 
\begin{tm}[Existence of multiple EE when $\mathcal{R}_0<1$]\label{T2-mass-action-EE-Multiple} Suppose that {\bf (A1)}-{\bf (A3)} holds, $\sum_{j\in\Omega} r_j<N$, and there exists $d_I>0$ such that $\mathcal{R}_0<1$.  Then there is $d^*>0$ such that  model \eqref{model-mass-action} has at least two EE solutions $(\bm S^{\rm min},\bm I^{\rm min})$ and $(\bm S^{\rm max},\bm I^{\rm max})$ for any $d_S\in (0, d^*)$, which satisfies $\bm I^{\rm min}<\bm I^{\rm max}$. Moreover, any other EE solution $(\bm S,\bm I)$ of \eqref{model-mass-action}, if exists, satisfies $\bm I^{\rm min}\ll\bm I\ll\bm I^{\rm max}$, and  the following statements hold:

\begin{itemize}    
    \item[\rm (i)]  $\bm S^{\rm max}\to\bm r$ and $\bm I^{\rm max}\to {(N-\sum_{j\in\Omega} r_j)}\bm\alpha$ as $d_S\to 0$;
    
      \item[\rm (ii)] $\bm S^{\rm min}\to \bm S^*$ and $\bm I^{\rm min} \to {\bf 0}$  as $d_S\to 0$, where $\bm S^*>{\bf 0}$ satisfies  $\sum_{j\in\Omega} S^*_j=N$ and   $\sigma_*(d_L\mathcal{L}+{\rm diag}(\bm \beta\circ(\bm S^*-\bm r)))=0$.
\end{itemize}  
\end{tm}
\begin{proof}

Define $\mathcal{F}: (0, \infty)\times(1/\mathcal{R}_0, \infty)\to (0, \infty)$ by 
\begin{equation}\label{N-d_I-d_S-def}
\mathcal{F}(d_S, l)=\mathcal{N}(l)+d_{S}l\sum_{j\in\Omega} U_j^{l},
\end{equation}
where $\bm U^l$ and $\mathcal{N}(l)$ are specified as in  Lemma \ref{lem2}.  It is easy to see that $\mathcal{F}$ is strictly increasing in $d_S$ for each fixed $l>1/\mathcal{R}_0$. By Lemma \ref{lem2}, the function $\mathcal{F}$ is continuously differentiable with
\begin{equation}\label{RT200}
\lim_{l\to\infty}\mathcal{F}(d_S, l)=\infty, \quad \forall\ d_S>0.
\end{equation}
By \eqref{lem2-eq2}, we have
$$
\lim_{l\to\infty}\mathcal{N}(l)=\sum_{j\in\Omega} r_j<N.
$$
Hence, there is $l^*> {1}/{\mathcal{R}_0}$ (depending on $d_I$  but independent of $d_S$) such that 
\begin{equation}\label{RT33}
\sup_{l\ge l^*}\mathcal{N}(l)<N.
\end{equation}
Denote 
$$ 
d^*:=\frac{N-\mathcal{N}(l^*)}{l^*\sum_{j\in\Omega}U^{l^*}_j},
$$
where $d^*$ is  depending on $d_I$ but independent of $d_S$.
By the monotonicity of $\mathcal{F}$ in $d_S$,  for any $d_S\in (0, d^*)$, we have
\begin{equation}\label{RT30}
\mathcal{F}(d_S, l^*)<\mathcal{F}(d^*, l^*)=\mathcal{N}(l^*)+\frac{N-\mathcal{N}(l^*)}{l^*\sum_{j\in\Omega} U^{l^*}_j}l^*\sum_{j\in\Omega} U_j^{l^*}=N.
\end{equation}
 By   \eqref{RT200} and \eqref{RT30}, for any $d_S\in (0, d^*)$, there is $l_{d_S}^{\max}>l^*$ such that 
\begin{equation}\label{RT31}
    \mathcal{F}(d_S, l_{d_S}^{\max})=N\quad \text{and}\quad \mathcal{F}(d_S, l)>N,\quad \forall\ l>l_{d_S}^{\max}.
\end{equation}

Since $\mathcal{F}$ is strictly increasing in $d_S$, $l^{\max}_{d_S}$ is strictly decreasing in $d_S$. By the equation 
in \eqref{RT31}, the monotonicity of  $\bm U^l$  in $l$, and  $l_{d_S}^{\max}>l^*$, we have
\begin{equation}\label{RT32}
l^{\max}_{d_S}=\frac{N-\mathcal{N}(l^{\max}_{d_S})}{d_S\sum_{j\in\Omega} U^{l^{\max}_{d_S}}}\ge \frac{N-\sup_{l\ge l^*}\mathcal{N}(l)}{d_S\sum_{j\in\Omega} U^{l^*}}\to \infty \quad \text{as}\ d_{S}\to 0.
\end{equation}
For any $d_S\in (0, d^*)$, define
$$
 (\bm S^{\max},\;\bm I^{\max} ):=\big(l^{\max}_{d_S}(N\bm\alpha -d_I\bm U^{l^{\max}_{d_S}}),\; d_Sl^{\max}_{d_S}\bm U^{l^{\max}_{d_S}}\big).
 $$
It follows from Lemma \ref{lem3}-{\rm (ii)} that $(\bm S^{\max},\,\bm I^{\max} )$ is an EE solution for any $d_S\in (0, d^*)$. Moreover by  \eqref{lem2-eq1} and \eqref{RT32},  
$$
\bm S^{\max}=l^{\max}_{d_S}(N\bm\alpha -d_I\bm U^{l^{\max}_{d_S}})\to \bm r\quad \text{as}\quad d_{S}\to 0.
$$
By \eqref{lem2-eq1}, \eqref{lem2-eq2} and  the equation 
in \eqref{RT31}, we have
$$
\bm I^{\max}=d_Sl^{\max}_{d_S}\bm U^{l^{\max}_{d_S}}=\frac{N-\mathcal{N}(l^{\max}_{d_S})}{\sum_{j\in\Omega} U_j^{l^{\max}_{d_S}}}\bm U^{l^{\max}_{d_S}}\to (N-\sum_{j\in\Omega} r_j)\bm\alpha \quad \text{as}\;\; d_S\to 0.
$$

Let  $(\bm S,\bm I)$ be another EE solution of \eqref{model-mass-action}. By Lemma \ref{lem3}, $(\bm S, \bm I)=(\bar l(N\bm\alpha-d_I\bm U^{\bar l}), d_S\bar l\bm U^{\bar l})$ for some $\bar l>1/\mathcal{R}$ satisfying $\mathcal{F}(d_S, \bar l)=N$. By \eqref{RT31}, we have $l_{d_S}^{\max}>\bar l$. By the strict monotonicity of $\bm U^l$ in $l$,  we have $\bm I\ll\bm I^{\rm max}$.  This completes the proof of {\rm (i)}.

 Next, we prove the existence of $(\bm S^{\rm min},\bm I^{\rm min} )$ and (ii). By \eqref{lem2-eq1}, \eqref{lem2-eq2} and  \eqref{RT30}, and the fact that $\mathcal{R}_0<1$,
 $$ 
\mathcal{F}(d_S, l^*)<N<\frac{N}{\mathcal{R}_0}=\lim_{l\to \frac{1}{\mathcal{R}_0}}\mathcal{F}(d_S, l),\quad \forall \,d_S\in (0, d^*).
$$
Thus for any $d_S\in (0, d^*)$, there is $l^{\rm min}_{d_S}\in({1}/{\mathcal{R}_0},l^*)$ such that
\begin{equation}\label{RT36}
    \mathcal{F}(d_S, l^{\rm min}_{d_S})=N\quad \text{and} \quad \mathcal{F}(d_S, l)>N, \quad \forall\ l\in ({1}/{\mathcal{R}_0}, l^{\rm min}_{d_S}).
\end{equation}
For any  $d_S\in (0, d^*)$, define
$$
 (\bm S^{\rm min},\,\bm I^{\rm min } ):=\big(l^{\rm min}_{d_S}(N\bm\alpha -d_I\bm U^{l^{\rm min}_{d_S}}),\; d_Sl^{\rm min}_{d_S}\bm U^{l^{\rm min}_{d_S}}\big).
 $$
Then $(\bm S^{\rm min},\,\bm I^{\rm min } )$ is an EE solution for any  $d_S\in (0, d^*)$. By \eqref{RT36}, using a similar argument as for  $(\bm S^{\rm max},\,\bm I^{\rm max} )$, we can show that  $(\bm S^{\rm min},\,\bm I^{\rm min } )$ is the minimal solution in the sense that $\bm I\ll \bm I^{\rm min}$ for any other EE solution $(\bm S, \bm I)$.

Similar to $l^{\rm max}_{d_S}$,  $l^{\rm min}_{d_S}$ is strictly decreasing in $d_S$. It follows  that 
$$ 
\underline{l}:={
\sup_{0<d_S<d^*}}l^{\rm min}_{d_S}=\lim_{d_S\to0}l^{\rm min}_{d_S}.
$$
By the equation in \eqref{RT36}, we have 
$$
\lim_{d_S\to0}\mathcal{N}(l^{\rm min}_{d_S})=\lim_{d_S\to0}( N-d_Sl^{\rm min}_{d_S}\sum_{j\in\Omega} U_j^{l^{\rm min}_{d_S}})=N<\frac{N}{\mathcal{R}_0}.
$$
So by \eqref{lem2-eq2}, we must have $\underline{l}>{1}/{\mathcal{R}_0}$. It follows that 
$$
\lim_{d_S\to 0}\bm S^{\rm min}=\bm S^*:=\underline{l}(N\bm \alpha-d_I\bm U^{\underline{l}})\quad \text{and} \quad \lim_{d_S\to 0}\bm I^{\rm min}=\bm 0,
$$ 
where $\bm S^*$ satisfies  $\sum_{j\in\Omega} S^*_j=\mathcal{N}(\underline{l})=N$. 
Moreover, $\bm S^*$ and $\bm U^{\underline{l}}$ are positive and satisfies
$$
0=d_I\mathcal{L}\bm U^{\underline{l}}+\bm \beta\circ(\bm S^*-\bm r)\circ\bm U^{\underline{l}},
$$
which implies that $\sigma_*(d_I\mathcal{L}+{\rm diag}(\bm \beta\circ(\bm S^*-\bm r)))=0$. This proves (ii). 
\end{proof}

Combining Theorem \ref{T1-mass-action-DEF-stability}-{\rm (ii)}, Proposition \ref{T1-mass-action-EE-Non-Existence} and Theorem \ref{T2-mass-action-EE-Multiple},  we have the following result when using $d_S$ as a bifurcation parameter:

\begin{tm}\label{theorem-ds}
Suppose that {\bf (A1)}-{\bf (A3)} holds, $\sum_{j\in\Omega} r_j<N$, and there exists $d_I>0$ such that $\mathcal{R}_0<1$.   Then there exist  $0< d_1^*\le d_2^*\le d_3^*$, depending on the parameters of the model except for $d_S$, such that \eqref{model-mass-action} has no EE solution if $d_S>d_2^*$, the DFE is globally stable if $d_S>d_3^*$, and \eqref{model-mass-action} has at least two EE solutions if  $0<d_S<d_1^*$.
\end{tm}

\begin{rk}
In Theorem \ref{theorem-ds}, it remains open   whether $d_1^*= d_2^*= d_3$. If $d_2^*<d_3^*$, it would be interesting to investigate  the large time behavior of the solutions when $d_S\in (d_2^*, d_3^*)$.  If $d_1^*<d_2^*$, it would be interesting to  determine the number of EE solutions when $d_S\in (d_1^*, d_2^*)$.   
\end{rk}  

\begin{rk}\label{R-1}
    Note that $\mathcal{R}_0$  is decreasing in $d_I$ and $\mathcal{R}_0\to \sum_{j\in\Omega} N\alpha_j^2\beta_j/\sum_{j\in\Omega}\alpha_j\gamma_j$ as $d_I\to\infty$. Hence if $\sum_{j\in\Omega} N\alpha_j^2\beta_j/\sum_{j\in\Omega}\alpha_j\gamma_j<1$, then there exists $d_I^*>0$ such that $\mathcal{R}_0<1$ for any $d_I>d_I^*$. Therefore, we can choose $N>0$ satisfying the conditions in Theorem \ref{theorem-ds}  if and only if $\sum_{j\in\Omega}\gamma_j/\beta_j<(\sum_{j\in\Omega}\alpha_j\gamma_j)/(\sum_{j\in\Omega}\beta_i\alpha_j^2)$. There are parameters $\bm \beta$ and $\bm\gamma$ satisfying the latter inequality. For example, if $ \gamma_j=(\beta_j\alpha_j)^2$ for all $j\in\Omega$ and $\bm \gamma\notin{\rm span}(\bm 1)$, then using  $\sum_{j\in\Omega}\alpha_j=1$ and the strict convexity of the function $t\mapsto t^2$, we obtain that 
    $$
\Big(\sum_{j\in\Omega}\frac{\gamma_j}{\beta_j}\Big)\Big(\sum_{j\in\Omega}\beta_j\alpha_j^2\Big)=\Big(\sum_{j\in\Omega}\beta_j\alpha_j^2\Big)^2=\Big(\sum_{j\in\Omega}\alpha_j(\beta_j\alpha_j)\Big)^2=\Big(\sum_{j\in\Omega}\alpha_j\sqrt{\gamma_j}\Big)^2<\sum_{j\in\Omega}\alpha_j\gamma_j.
    $$
\end{rk}

\subsection{Case $\mathcal{R}_0>1$}
Now we assume that $\mathcal{R}_0>1$ and find a parameter range such that the EE solution is unique. 

\begin{tm}[Existence and uniqueness of EE]\label{T2-mass-action-EE-Existence} Suppose that {\bf (A1)}-{\bf (A3)} holds and $d_I$ is fixed. 
\begin{itemize}
    \item[\rm (i)] If $\mathcal{R}_0>1$, then \eqref{model-mass-action} has at least one EE solution for every $d_S>0$. 
    \item[\rm (ii)] There is  $m_{d_I}\in(0,1]$, independent of $N$   and depending on $d_I$, such that if $d_S>(1-m_{d_I})d_I$, then \eqref{model-mass-action} has a unique EE solution if $\mathcal{R}_0>1$ and no EE solution if $\mathcal{R}_0\le 1$. Moreover, $m_{d_I}=1$ if  $\bm\gamma=m\bm\beta\circ\bm\alpha$ for some $m>0$.
\end{itemize}  
\end{tm}
\begin{proof} 
(i) Let $\mathcal{F}$ be defined as in the proof of Theorem \ref{T2-mass-action-EE-Multiple}. For any $d_S>0$, it follows from \eqref{lem2-eq1} and \eqref{lem2-eq2} that 
$$
\lim_{l\to\frac{1}{\mathcal{R}_0}}\mathcal{F}(d_S, l)= \frac{N}{\mathcal{R}_0}<N.
$$
Hence by \eqref{RT200}, for any $d_S>0$, there is $l^{d_S}>{1}/{\mathcal{R}_0}$ satisfying $\mathcal{F}(d_S, l^{d_S})=N$. So \eqref{N-equation} holds for $l=l^{d_S}$. Therefore by lemma \ref{lem3}-{\rm (ii)}, $(\bm S,\bm I)=(l^{d_S}(N\bm\alpha -\bm U^{l^{d_S}}),d_Sl^{d_S}\bm U^{l^{d_S}})$ is an EE solution of \eqref{model-mass-action}. 

 (ii) We  show that there is $0<m_{d_I}\le 1$, independent of  $N$, such that 
\begin{equation}\label{PLL2}
    \frac{\partial\mathcal{F}(d_S,l)}{\partial l}>0, \quad \forall \ l>\frac{1}{\mathcal{R}_0}, \ d_S>(1-m_{d_I})d_I,\ N>0. 
\end{equation}
Note that if \eqref{PLL2} holds, then  for any $N>0$ and $d_S>(1-m_{d_I})d_I$, the algebraic equation $\mathcal{F}(d_S,l)=N$ in $l$ has at most one solution in $(1/\mathcal{R}_0,\infty)$. Furthermore, this algebraic equation has a solution if and only if $\mathcal{R}_0>1$. So by Lemma \ref{lem3}, if we can prove \eqref{PLL2} then the existence and uniqueness results for EE solutions in (ii) hold.

To prove \eqref{PLL2}, we  compute that 
\begin{align}\label{RT21}
\frac{\partial\mathcal{F}(d_S, l)}{\partial l}=&\sum_{j\in\Omega}N\alpha_j+(d_S-d_I)l\sum_{j\in\Omega}\frac{d U_j^l}{dl}+(d_S-d_I)\sum_{j\in\Omega} U_j^l\cr 
=&N+(d_S-d_I)\Big(l\sum_{j\in\Omega}\frac{d U_j^l}{dl}+\sum_{j\in\Omega} U^l_j\Big), \quad \quad\forall\ l>\frac{1}{\mathcal{R}_0},\ d_S>0, \ N>0.
\end{align}
For each $N>0$, define
$$
\mathcal{M}_{d_I}{(N)}:=\sup_{l>\frac{1}{\mathcal{R}_0}}\Big(l\sum_{j\in\Omega}\frac{d U_j^l}{dl}+\sum_{j\in\Omega} U^l_j\Big).
$$
Since $\bm U^l$ is strictly increasing in $l$, we have $\mathcal{M}_{d_I}{(N)}\in (0,\infty]$. If $ \mathcal{M}_{d_I}{(N)}<\infty$,  by \eqref{RT21}, we see that 
$$
\frac{\partial\mathcal{F}(d_S, l)}{\partial l}>0, \quad \quad \forall\,d_S>\Big(1-\frac{N}{d_I\mathcal{M}_{d_I}{(N)}}\Big)_+d_{I}.
$$
Hence, the EE solution is unique if $d_S>\Big(1-\frac{N}{d_I\mathcal{M}_{d_I}{(N)}}\Big)_+d_{I}.$ 

Next, we show that $\mathcal{M}_{d_I}{(N)}$ is finite.  Since $l\mapsto \mathcal{K}_{d_I}(l):=l\sum_{j\in\Omega}\frac{d\ U_j^l}{dl}+\sum_{j\in\Omega} U^l_j$ is continuous, it suffices to show that 
\begin{equation}\label{RT22}
   \limsup_{l\to\infty}\mathcal{K}_{d_I}(l)<\infty
\end{equation}
and 
\begin{equation}\label{RT23}
   \limsup_{l\to\frac{1}{\mathcal{R}_0}}\mathcal{K}_{d_I}(l)<\infty.
\end{equation}

We first show  \eqref{RT22}. Differentiate \eqref{Eq1} with respect to $l$, we obtain 
\begin{equation}\label{RT25}
    0=d_I\mathcal{L}\frac{d \bm U^l}{dl}+(l\bm\beta\circ(N\bm\alpha -d_I\bm U^l)-\bm\gamma)\circ \frac{d\bm U^l}{dl} + \bm\beta\circ\Big( (N\bm\alpha-d_I\bm U^l)-l d_I\frac{d\bm U^l}{dl}\Big)\circ \bm U^l,\quad l>\frac{1}{\mathcal{R}_0}.
\end{equation}
Let $\bm V^l=l\frac{d\bm U^l}{dl}$. Then $\bm V^l$  is the unique positive solution of 
\begin{equation}\label{RT26}
    \bm 0=d_I\mathcal{L}\bm V^l+(l\bm\beta\circ(N\bm\alpha-d_I\bm U^l)-\bm\gamma)\circ\bm V^l+l\bm\beta\circ((N\bm\alpha-d_I\bm U^l)-d_I\bm V^l)\circ\bm U^l,\quad l>\frac{1}{\mathcal{R}_0}.
\end{equation}
Using the fact that $\bm U^l$ solves \eqref{Eq1}, we can check that $ \frac{\|N\bm\alpha -d_I\bm U^l\|_{\infty}}{d_I \bm U^l_m}\bm U^l$ is a super solution of \eqref{RT26}. Hence,
 $$
 \bm V^l\le \frac{\|N\bm\alpha -d_I\bm U^l\|_{\infty}}{d_I \bm U^l_m}\bm U^l,  \quad\forall\, l>\frac{1}{\mathcal{R}_0}.
 $$
Since $d_I\bm U^l\to N\bm\alpha$ as $l\to\infty$ by \eqref{lem2-eq1}, we conclude that $\|\bm V^l\|_1\to 0$ as $l\to \infty$.  This implies that 
\begin{equation}\label{RT27}
\lim_{l\to\infty}\mathcal{K}_{d_I}(l)=\lim_{l\to\infty}\sum_{j\in\Omega}( V^l_j+ U^l_j)=\frac{N}{d_I},
\end{equation}
which yields  \eqref{RT22}. 

Let $\bm Z^l=l(N\bm\alpha-d_I\bm U^l)$ is as in \eqref{Z-l=eq}.  Then by \eqref{Z-l-bounds},  we have
$$
\|l\bm\beta\circ(N\bm\alpha-d_I\bm U^l)-\bm\gamma\|_1=\|\bm\beta\circ(\bm Z^l-\bm r)\|_1\le \|\bm\beta\|_{\infty}\|\bm Z^l-\bm r\|_1 \le \|\bm \beta\|_\infty \left(\frac{\bm r_M}{\bm\alpha_m}+\|\bm r\|_1 \right),\quad \forall\ l>\frac{1}{\mathcal{R}_0}.
$$
It follows from Lemma \ref{Harnack-eq1} that there is $c_{d_I}>0$, {independent of $N$,} such that 
$$
\|\bm U^l\|_{\infty}\le c_{d_I}\bm U_m^l,\quad \forall\ l>\frac{1}{\mathcal{R}_0}.
$$
Therefore, 
\begin{equation}\label{PLL1}
\sum_{j\in\Omega} V^l_j\le \frac{\|N\bm\alpha-d_I\bm U^l\|_1}{d_I U^l_m}\sum_{j\in\Omega} U^l_j\le \frac{|\Omega|c_{d_I}}{d_I}\|N\bm\alpha-d_I\bm U^l\|_1, \quad \forall\ l>\frac{1}{\mathcal{R}_0}.
\end{equation}
As a result, 
$$
\limsup_{l\to\frac{1}{\mathcal{R}_0}}\mathcal{K}_{d_I}(l)\le \frac{c_{d_I}|\Omega| }{d_I}\lim_{l\to\frac{1}{\mathcal{R}_0}}\|N\bm\alpha-d_I\bm U^l\|_1=\frac{c_{d_I}N|\Omega|}{d_I}.
$$
This proves \eqref{RT23}. Therefore, $\mathcal{M}_{d_I}{(N)}<\infty$.  Moreover, by \eqref{RT27}, $d_I\mathcal{M}_{d_I}{(N)}\ge N$, which is equivalent to $0<\tilde m_{d_I}(N):=\frac{N}{d_I\mathcal{M}_{d_I}{(N)}}\le 1$. {Moreover, we have from \eqref{PLL1} and the fact that $\bm U^l\le \frac{N}{d_I}\bm \alpha$ (see \eqref{lem2-eq1}) that 
$$
\mathcal{K}_{d_I}(l)=\sum_{j\in\Omega}(U_j^l+V_j^l)\le \sum_{j\in\Omega}N\alpha_j+\frac{|\Omega|c_{d_I}}{d_I}\sum_{j\in\Omega}N\alpha_j=N\Big(1+\frac{|\Omega|c_{d_I}}{d_I}\Big).
$$
Therefore, $$\frac{d_I\mathcal{M}_{d_I}(N)}{N}=\frac{d_I}{N}\sup_{l>\frac{1}{\mathcal{R}_0}}\mathcal{K}_{d_I}(l)\le d_I+|\Omega|c_{d_I}. $$ Hence,   setting  $m_{d_I}:=\inf_{N>0}\tilde m_{d_I}(N)$, 
we have that $0<\frac{1}{d_I+|\Omega|c_{d_I}}\le m_{d_I}\le 1$, $m_{d_I}$ is independent of $N$, and $\partial_l\mathcal{F}(d_S,l)>0$ for any $d_S>(1-m_{d_I})d_I$, $l>1/\mathcal{R}_0$ and $N>0$. 
} 
This completes the assertion about the uniqueness of EE when $d_S>(1-m_{d_I})d_I$. 

Finally, suppose that $\bm\gamma=m\bm\beta\circ\bm\alpha$ for some positive number $m$.  In this case, it is easy to check that the unique positive solution of \eqref{Eq1} is $\bm U^l=\frac{1}{d_I}(N-\frac{m}{l})\bm \alpha$ for all $l>{1}/{\mathcal{R}_0}$. Hence
$$
\mathcal{M}_{d_l}(l)=\sum_{j\in\Omega}\Big(\frac{m}{ld_I}+\frac{1}{d_I}\Big(N-\frac{m}{l}\Big)\Big)\alpha_j=\frac{N}{d_I},\quad \forall\ l>\frac{1}{\mathcal{R}_0},
$$
which gives $m_{d_I}=1$. 
\end{proof}

We conjecture that if the EE solution is unique it is globally stable. However, we are only able to prove this for some special parameter ranges:

\begin{tm}[Stability of EE when $\mathcal{R}_0>1$]\label{T1-mass-action-EE-stability} Suppose that ${\bf (A1)}$-${\bf (A3)}$ holds and $\mathcal{R}_0>1$.  Then the $\bm I$-component of the solution is uniformly persistent in the sense that there is $\eta_*>0$ independent of initial data such that 
\begin{equation}\label{persistence-eq}
    \liminf_{t\to\infty}\min_{j\in\Omega} I_{j}(t)\ge \eta_*
\end{equation}
for any solution $(\bm S(t),\bm I(t))$ of \eqref{model-mass-action} with $\bm I^0\neq \bm 0$. 
Moreover, the following conclusions hold:
\begin{itemize}
 \item [\rm (i)] If $\bm\gamma=m\bm\beta\circ\bm \alpha$ for some positive number $m>0$, then any solution of \eqref{model-mass-action} with $\bm I^0\neq \bm 0$ converges to $ (\bm r, (N-\sum_{j\in\Omega} r_j)\bm \alpha)$ as $t\to\infty$;
 
\item[\rm (ii)] If $d_S=d_I$, then any solution of \eqref{model-mass-action} with $\bm I^0\neq \bm 0$ converges to $({N}\bm\alpha-{\bm I_0},{\bm I_0})$ as $t\to\infty$, where ${\bm I_0}$ is the unique positive solution of 
    \begin{equation}\label{I-tilde-equation}
    \bm 0=d_I\mathcal{L}\tilde{\bm I}+\Big(\bm\beta\circ(N\bm \alpha-\tilde{\bm I})-\bm\gamma\Big)\circ\tilde{\bm I}.
    \end{equation}
   
\end{itemize}
\end{tm}

\begin{proof}
The uniform persistence of solutions  when $\mathcal{R}_0>1$ is standard  (e.g., see the proof of \cite[Theorem 2.3]{wang2004epidemic}), so we omit the proof here.

 Let $(\bm S(t),\bm I(t))$ be a solution of \eqref{model-mass-action} with  $\bm I^0\neq \bm 0$.

\noindent{\rm (i)} Suppose that $\bm \gamma=m\bm\beta\circ\bm \alpha$ for some $m>0$.  By \eqref{Sl0} we have that 
\begin{equation}\label{RT42}
m\bm\alpha -me^{-\bm\beta_m\int_{1}^t \bm I_m(s)ds}\bm \alpha\le \bm S(t),\quad t\ge 1.
\end{equation}
    On the other hand, define 
    $$
    {\bm S}^{\rm up}(t)=m\bm\alpha+\frac{\|\bm S(1)\|_{\infty}}{\bm\alpha_m}e^{- \bm\beta_{m}\int_{1}^t \bm I_m(s)ds}\bm \alpha, \quad \forall\, t>1.
    $$
 Hence, 
    \begin{align*}
         \frac{d {\bm S}^{\rm up}}{dt}-d\mathcal{L}{\bm S}^{\rm up}-\bm \beta\circ(m\bm\alpha-{\bm S}^{\rm up})\circ \bm I =&\frac{\|\bm S(1)\|_{\infty}}{\bm\alpha_m}e^{-\bm\beta_m\int_{1}^t \bm I_m(s)ds}\Big(-\bm\beta_m \bm I_m(t){\bm 1} +\bm\beta\circ\bm I \Big)\circ\bm \alpha\ge {\bm 0}. 
    \end{align*}
 Noticing $\bm S^{\rm up}(1)\ge \bm S(1)$, by the comparison principle for monotone dynamical systems, we have
    \begin{equation}\label{RT43}
        \bm S(t)\le  {\bm S}^{\rm up}(t)=m\bm\alpha+\frac{\|\bm S(1)\|_{\infty}}{\bm\alpha_m}e^{- \bm\beta_{m}\int_{1}^t \bm I_m(s)ds}\bm \alpha, \quad \forall\, t>1.
    \end{equation}
By \eqref{persistence-eq}, we have $ \int_{1}^{\infty} \bm I_m(t)dt=\infty$. As a result, we conclude from \eqref{RT42}-\eqref{RT43} that 
    \begin{equation}\label{RT44}
    \bm S(t)\to m\bm \alpha=\bm r\quad \text{as}\quad  t\to\infty.
    \end{equation}
    This in turn implies that 
    $$
    \|(\bm\beta\circ\bm S(t)-\bm\gamma)\circ \bm I(t)\|_{\infty}\to 0 \quad \text{as}\quad t\to\infty.
    $$
    Therefore, by Lemma \ref{lem0}, 
    \begin{equation}\label{RT45}
        \big\|\bm I(t)-\bm\alpha{\sum_{j\in\Omega} I_j(t)}\big\|_1\to 0\quad \text{as}\quad t\to\infty.
    \end{equation}
   By \eqref{RT44}, we have 
    $$
   {\sum_{j\in\Omega} I_j(t)}={N}-{\sum_{j\in\Omega} S_j(t)}\to {N-\sum_{j\in\Omega} r_j}\quad \text{as}\quad t\to\infty.
    $$
    It then follows from \eqref{RT44}-\eqref{RT45} that $(\bm S(t),\bm I(t))\to \Big(\bm r, ({N-\sum_{j\in\Omega} r_j})\bm \alpha\Big)$ as $t\to\infty$.

    \medskip

    \noindent{\rm (ii)} Suppose that $d_S=d_I$. Similar to the proof of Theorem \ref{T1-mass-action-DEF-stability}, 
 $\bm Z(t):=\bm S(t)+\bm I(t)$ satisfies \eqref{RT7}. And,
    \begin{equation}\label{RT48-b}
        \limsup_{t\to\infty} I_j(t)\le  I_{0j},\quad \forall\ j\in\Omega,
    \end{equation}
    where $\bm I_{0}$ is the unique positive solution of \eqref{RT9} with $\varepsilon=0$. By \eqref{RT7-2},  
    \begin{equation}\label{RT48}
        \frac{d\bm I}{dt}\ge d_I\mathcal{L}\bm I+\Big(\bm\beta\circ\Big( (1-\varepsilon)N\bm\alpha-\bm I(t)\Big)-\bm\gamma)\circ \bm I,\quad \forall\ t>t_{\varepsilon}.
    \end{equation}
    Since $\mathcal{R}_0>1$, we have $\sigma_*(d_I\mathcal{L}+{\rm diag}(N\bm\beta\circ\bm\alpha-\bm\gamma))>0$. Noticing
    $$
    \sigma_*(d_I\mathcal{L}+{\rm diag}(N\bm\beta\circ\alpha-\bm\gamma))\to \sigma_*(d_I\mathcal{L}+{\rm diag}((1-\varepsilon)N\bm\beta\circ\alpha-\bm\gamma))\quad \text{as}\ \varepsilon\to0,
    $$
   we have 
    $$\sigma_*(d_I\mathcal{L}+{\rm diag}((1-\varepsilon)N\bm\beta\circ\alpha-\bm\gamma))>0\quad \text{for}\; 0<\varepsilon\ll 1.$$
    Thus, there is $0<\varepsilon_0\ll 1$ such that the system 
    \begin{equation}\label{RT49}
        \bm 0=d_I\mathcal{L}\bm I^{\varepsilon,{\rm low}} +\Big(\bm\beta\circ\Big( (1-\varepsilon)N\bm\alpha-\bm I^{\varepsilon,{\rm low}}\Big) -\bm\gamma\Big)\circ\bm I^{\varepsilon,{\rm low}} 
    \end{equation}
    has a unique positive and stable solution $\bm I^{\varepsilon,{\rm low}}$. Observe that $\bm I^{\varepsilon,{\rm low}}$ is decreasing in $\varepsilon$. Hence $\bm I^{\varepsilon,{\rm low}}\to \bm I_0$ as $\varepsilon\to 0$, where $\bm I_0$ is the unique positive solution of \eqref{RT9} when $\varepsilon=0$. By \eqref{RT48}-\eqref{RT49} and the comparison principle for monotone dynamical systems, we have that 
    $$
    \liminf_{t\to\infty} I_j(t)\ge  I_j^{\varepsilon,{\rm low}}, \quad \forall\ j\in\Omega.
    $$
    Taking $\varepsilon\to 0$, we obtain 
    $$
    \liminf_{t\to\infty} I_j(t)\ge  I_{0j},\quad \forall\ j\in\Omega.
    $$
    This together with \eqref{RT48-b} yields that $\bm I(t)\to \bm I_0$ as $t\to\infty$. Noticing \eqref{RT7},
    $ 
    (\bm S(t),\bm I(t))\to (N\bm\alpha-\bm I_0,\bm I_0)
    $ as $t\to\infty$.
 \end{proof}
 \begin{rk}
In the absence of the disease,  the  population will ideally be distributed according to the DFE (see Lemma \ref{lem0}), in which case the susceptible population will be distributed proportionally to $\bm\alpha$. If the distribution of the risk function is proportional to $\bm\alpha$, hence proportional to the distribution of the susceptible population at the DFE, then the relative  dispersal rates $d_S$ and $d_I$ of the infected and susceptible individuals  become insignificant  and the global dynamics is governed by the local dynamics. In this case, the local reproduction number, defined as $\mathcal{R}_i:=\frac{N\beta_i\alpha_i}{\gamma_i}=\frac{N\alpha_i}{r_i}$, $i\in\Omega$, is patch independent and $\mathcal{R}_0=\mathcal{R}_i$ for all $i\in\Omega$. In this case, Theorem \ref{T1-mass-action-DEF-stability}-{\rm (iv)} and Theorem \ref{T1-mass-action-EE-stability}-{\rm (i)} confirm that the global dynamics of the disease is solely determined by the sign of $\mathcal{R}_0-1$. When both the susceptible and infected populations have the same dispersal rate, that is $d=d_I=d_S$, Theorem \ref{T1-mass-action-DEF-stability}-{\rm (iii)} and Theorem \ref{T1-mass-action-EE-stability}-{\rm (ii)} suggest that the global dynamics of the disease is again determine by  the sign of $\mathcal{R}_0-1$. Note in the latter case that  $\mathcal{R}_0$ depends on $d$.
\end{rk}

\section{Asymptotic profiles of EE solutions}

In this section, we consider the asymptotic profiles of the EE as $d_S$ and/or $d_I$ approach zero. 

\subsection{The case of $d_S\to 0$}
Firstly, we consider the case that $d_I$ is fixed and $d_S$ is approaching zero. 
 
 \begin{tm}\label{thm-ee-ds}Suppose that {\bf (A1)}-{\bf (A3)} holds and $d_I>0$ is fixed.  If \eqref{model-mass-action} has an EE solution for some $d_S=d_S^0>0$, then \eqref{model-mass-action} has at least one EE solution for any $d_S\in (0, d_S^0)$.  Moreover, the following statements about the EE  of \eqref{model-mass-action} hold:
 \begin{itemize}
     \item[\rm (i)] 
     If $N< \sum_{i\in\Omega}r_i$,   the EE $(\bm S, \bm I)$ of \eqref{model-mass-action} satisfies that $\bm I\to \bm {\bm 0}$, and up to a subsequence $\bm S\to \bm S^*$ as $d_S\to 0$, where  $\bm S^*=(N+d_I\sum_{i\in\Omega}\bar I_i^*)\bm\alpha-d_I\bar {\bm I}^*$  and $\bar {\bm I}^*$ is a nonnegative solution of  
\begin{equation}\label{barIs}
d_I\sum_{j\in\Omega}L_{ij}\bar I_j^*+ \beta_i\left(\alpha_iN-r_i+d_I\alpha_i\sum_{i\in\Omega}\bar I^*_i-d_I\bar I_i^* \right)\bar I_i^*=0, \ \ i\in\Omega.
\end{equation}

\item[\rm (ii)]  If $N=\sum_{i\in\Omega}r_i$, then  the EE $(\bm S, \bm I)$ of \eqref{model-mass-action} satisfies that $\bm I\to \bm {\bm 0}$ as $d_S\to 0$. Furthermore, up to a subsequence, either $\bm S\to \bm r$ or $\bm S\to \bm S^*$ as $d_S\to 0$, where  $\bm S^*$ is the same as {\rm (i)}. 

     \item[\rm (iii)] If $N> \sum_{i\in\Omega}r_i$, then \eqref{model-mass-action} has an EE $(\bm S, \bm I)$  for any $0<d_S\ll 1$ such that $(\bm S, \bm I)\to (\bm S^*, \bm I^*)$ as $d_S\to 0$, where
     $(\bm S^*, \bm I^*)=(\bm r, (N-\sum_{i\in\Omega}r_i)\bm \alpha)$.
 \end{itemize}
\end{tm}
\begin{proof}  Let $\mathcal{F}(d_S,l)$ be defined as in \eqref{N-d_I-d_S-def}. Since $(\bm S(\cdot,d_S^0),\bm I(\cdot,d_S^0))$ is an EE solution of \eqref{model-mass-action}, we have 
$$
\mathcal{F}\Big(d_S^0,\frac{\kappa^*_{d_S^0}}{d_S^0}\Big)=N,
$$
where $\kappa^*_{d_S^0}>{d_S^0}/{\mathcal{R}_0}$ is defined as in  Lemma \ref{lem3}. Since the mapping $d_S\mapsto \mathcal{F}(d_S,l)$ is strictly increasing in $d_S$ for any $l\in({1}/{\mathcal{R}_0}, \infty)$,  
$$
\mathcal{F}\Big(d_S,\frac{\kappa^*_{d_S^0}}{d_S^0}\Big)<N,\quad \forall\ 0<d_S<d_S^0.
$$ 
Combining this with \eqref{RT200}, for any $0<d_S<d_S^0$, there is $l^{d_S}>{\kappa^*_{d_S^0}}/{d_S^0}$ such that $\mathcal{F}(d_S,l^{d_S})=N$. Hence by Lemma \ref{lem3}-{\rm (ii)}, $(\bm S,\bm I)=\big(l^{d_S}(N\bm\alpha-d_I\bm U^{l^{d_S}}),d_Sl^{l^{d_S}}\bm U^{l^{d_S}}\big)$ is an EE solution of \eqref{model-mass-action} for any $d_S\in (0, d_S^0)$.


Let $(\bm S,\bm I)$ be an EE solution.
Restricting to a subsequence if necessary, we have $(\bm S, \bm I)\to (\bm S^*, \bm I^*)$ as $d_S\to 0$, where $(\bm S^*, \bm I^*)$ satisfies 
\begin{equation}\label{eeds}
 \left\{
\begin{array}{lll}
-\beta_i S_i^*I_i^*+\gamma_iI_i^*=0, &\quad j\in\Omega,\\
d_I\sum_{j\in\Omega}L_{ij}I_j^*+\beta_i S_i^*I_i^*-\gamma_iI_i^*=0, &\quad j\in\Omega,\\
\sum_{i\in\Omega}(S_i^*+I_i^*)=N. &
\end{array}
 \right.
\end{equation}
Adding up the first two equations of \eqref{eeds}, we obtain $\sum_{j\in\Omega}L_{ij}I_j^*=0$ for all $i\in\Omega$, which implies that $\bm I^*=k\bm \alpha$ for some $k\ge 0$. By the first equation of \eqref{eeds}, either $S_i^*=r_i$ or $I_i^*=0$ for all $i\in\Omega$. If $k>0$, then $\bm S^*=\bm r$. By the third equation of \eqref{eeds}, $k=N-\sum_{i\in\Omega}r_i$.

{\rm (i)} Suppose that $N< \sum_{i\in\Omega}r_i$. Then $k=0$ and $\bm I\to \bm 0$ as $d_S\to 0$. It follows Lemma \ref{lem3} that $\kappa^*=d_S+ \frac{(d_I-d_S)}{N}\sum_{i\in\Omega}I_i\to 0$ as $d_S\to 0$. We claim that 
\begin{equation}\label{kappa-bond}
\limsup_{d_S\to 0}\frac{\kappa^*}{d_S}<\infty.
\end{equation}
Assume to the contrary that the claim is false. Then restricting to a subsequence of $d_{S}\to 0$ if necessary $l:=\kappa^*/d_{S}\to \infty$. By Lemmas \ref{lem2}-\ref{lem3}, 
\begin{equation}\label{S-limit}
\lim_{d_S\to 0} \|\bm S-\bm r\|_1=\lim_{l\to\infty}\|l(N\bm \alpha-d_I\bm U)-\bm r\|_1=0.
\end{equation}
This implies that $N=\lim_{d_S\to 0} \sum_{i\in\Omega}(S_{i}+I_{i})=\sum_{i\in\Omega} r_i$, which is a contradiction. This proves the claim. 

Denote $\bar {\bm I}=\bm I/d_S=\tilde {\bm I} \kappa^*/d_S$. By Lemma \ref{lem2}, we have $\tilde {\bm I}<N\bm\alpha/d_I$. Hence, $\bar {\bm I}$ is bonded above as $d_S\to 0$. 
Since $S_i=(\kappa^*N\alpha_i-d_II_i)/d_S=(N+(d_I-d_S)\sum_{i\in\Omega}\bar I_i)\alpha_i-d_I\bar I_i$, the following equation holds:
$$
d_I\sum_{j\in\Omega}L_{ij}\bar I_j+\left(\beta_i(N+(d_I-d_S)\sum_{i\in\Omega}\bar I_i)\alpha_i-d_I\bar I_i-\gamma_i\right)\bar I_i=0, \ \ i\in\Omega.
$$
Taking $d_S\to0$, restricting to a subsequence if necessary, we may assume $\bar {\bm I}\to \bar{\bm I}^*$, where $\bar{\bm I}^*$ satisfies \eqref{barIs}. Then $\bm S\to \bm S^*$ as $d_S\to 0$, where $\bm S^*=(N+d_I\sum_{i\in\Omega}\bar I_i^*)\bm\alpha-d_I\bar {\bm I}^*$.

{\rm (ii)} Suppose that $N=\sum_{i\in\Omega}r_i$. Then $k=0$ and $\bm I\to 0$ as $d_S\to 0$.  Now, we distinguish two cases. In the first case, suppose that \eqref{kappa-bond} holds. Then as in {\rm (i)}, up to a subsequence,  $\bm S\to \bm S^*$ as $d_S\to 0$,   where  $\bm S^*=(N+d_I\sum_{i\in\Omega}\bar I_i^*)\bm\alpha-d_I\bar {\bm I}^*$ and  $\bar{\bm I}^*$ is a nonnegative solution of \eqref{barIs}. In the second case, suppose that \eqref{kappa-bond} does not hold.  Then restricting to a subsequence if necessary $l:=\kappa^*(d_{S})/d_{S}\to \infty$ as $d_S\to 0$. It then follows from Lemmas \ref{lem2}-\ref{lem3} that \eqref{S-limit} holds. 

{\rm (iii)} Suppose that $N>\sum_{i\in\Omega}r_i$. Then there is $l^*>\frac{1}{\mathcal{R}_0}$ such that \eqref{RT33} holds. Therefore, the EE solution $(\bm S^{\rm max},\bm I^{\rm max})$ of \eqref{model-mass-action} for $0<d_S\ll 1$ of Theorem \ref{T2-mass-action-EE-Multiple}-{\rm (ii-1)} are well defined and satisfy the desired asymptotic profiles as $d_S\to 0$. 
\end{proof}

When $\bm r\in{\rm span}(\bm \alpha)$, it follows from Theorem \ref{T1-mass-action-EE-stability}-{\rm (i)} that \eqref{model-mass-action} has a (unique) EE solution given by $(\bm r,(N-\sum_{j\in\Omega}r_j)\bm \alpha)$ for every $d_S>0$ and $d_I>0$ if $\mathcal{R}_0>1$. In this case, the EE is independent of the dispersal rates of the population, and the conclusion of Theorem \ref{thm-ee-ds}-{\rm (iii)} holds along any sequence of $d_S$ converging to zero. If $\bm r\notin{\rm span}(\bm \alpha)$, $\mathcal{R}_0>1$ and $N>\sum_{j\in\Omega}r_j$, Theorem \ref{thm-ee-ds}-{\rm (iii)} establishes the convergences of the EE solutions as $d_S$ tends to zero along a subsequence. Our next result shows that this convergence  holds along all sequences of $d_S$ converging to zero if the total population size exceeds some critical number.

\begin{prop}\label{thm-ee-ds-2} Suppose that {\bf (A1)}-{\bf (A3)} holds and $\bm r\notin {\rm span}(\bm \alpha)$, and fix $d_I>0$. Then there exists $ N^*_{d_I}>0$ with
$$
\max\Big\{\sum_{i\in\Omega}r_i,\frac{1}{\rho({\rm diag}(\bm\alpha\circ\bm\beta)V^{-1})}\Big\}\le N^*_{d_I}<(\bm r/\bm \alpha)_{M}
$$ 
such that   for any $N>N^*_{d_I}$, we have $\mathcal{R}_0>1$. Moreover, the EE solution $(\bm S,\bm I)$ of \eqref{model-mass-action}  satisfies $(\bm S,\bm I)\to (\bm r,(N-\sum_{j\in\Omega}r_j)\bm\alpha)$ as $d_S\to 0$.
    
\end{prop}
\begin{proof} 
Note from \eqref{F}-\eqref{R-0} that $\mathcal{R}_0=N\rho({\rm diag}(\bm\alpha\circ\bm\beta)V^{-1})$. Hence, we have $\rho({\rm diag}(\bm\alpha\circ\bm\beta)V^{-1})={\mathcal{R}_0}/{N}$, and 
$$
l>\frac{1}{\mathcal{R}_0}\quad \text{if and only if}\quad lN>R_*:=\frac{1}{\rho({\rm diag}(\bm\alpha\circ\bm\beta)V^{-1})}.
$$ 
So by Lemma \ref{lem2}, \eqref{Eq1} has a positive solution if and only if $lN>R_*$. 

Let $\mathcal{N}$ be defined by \eqref{N-d_I-def}. To emphasize the dependence of $\mathcal{N}$ on $N$, we will write it as $\mathcal{N}(l, N)$, which is defined if $lN>R_*$. 
Let 
\begin{equation}\label{N-critic-def}
    N^*_{d_I}:=\sup_{lN>R_*}\mathcal{N}(l,N).
\end{equation}
We claim that 
\begin{equation}\label{N-critic-bounds}
    \max\Big\{\sum_{j\in\Omega}r_j,R_*\Big\}\le N^*_{d_I}<(\bm r/\bm \alpha)_M.
\end{equation}
It is clear from \eqref{lem2-eq2} that the lower bound for $N^*_{d_I}$ in \eqref{N-critic-bounds} holds. Since $\bm r\notin {\rm span}(\bm \alpha)$, it is easy to see that  $(\bm r/\bm\alpha)_M\bm\alpha$ is a strict supersolution of \eqref{Z-l=eq}. Hence,
\begin{equation}\label{DFF1}
   \bm Z^{l}\ll (\bm r/\bm \alpha)_M\bm\alpha,\quad \forall\ lN>R_*,
\end{equation}
where $\bm Z^{l}$ is the solution of \eqref{Z-l=eq}.
It follows that 
\begin{equation}\label{DFF2}
    \mathcal{N}(l,N)=\sum_{j\in\Omega}Z^l_j<(\bm r/\bm\alpha)_M,\quad \forall\ lN>R_*.
\end{equation}
Thus, $N^*_{d_I}\le(\bm r/\bm\alpha)_M$. 
To complete the proof of \eqref{N-critic-bounds}, it remains to show that the last inequality is strict.  Since $\bm r\notin {\rm span}(\bm\alpha)$, it follows from \cite[Theorems 2.6 and 2.7]{chen2020asymptotic} that $\rho({\rm diag}(\bm\beta\circ\bm\alpha)V^{-1})$ is strictly decreasing in $d_I$. Hence by \eqref{R-0-limit}, we have that
\begin{equation}\label{DF3}
    R_*=\frac{1}{ \rho({\rm diag}(\bm\beta\circ\bm\alpha)V^{-1})}<\frac{\sum_{j\in\Omega}\alpha_j\gamma_j}{\sum_{j\in\Omega}\alpha_j^2\beta_j}=\frac{\sum_{j\in\Omega}\alpha_j^2(r_j/\alpha_j)\beta_j}{\sum_{j\in\Omega}\alpha_j^2\beta_j}<(\bm r/\bm\alpha)_M.
\end{equation}
Let $\{(l^m,N^m)\}_{m\ge 1}$ be a sequence of positive numbers satisfying $l^mN^m>R_*$ for every $m\ge 1$ such that 
$$ 
\lim_{m\to\infty}\mathcal{N}(l^m,N^m)=N^*_{d_I}.
$$
Without loss of generality, we may suppose that $l^mN^m\to L\in[R_*,\infty]$ as $m\to\infty$.  We distinguish two cases.

\noindent{\bf Case 1.} $L=\infty$. Observe  that $\underline{\bm U}^{l}:=\frac{1}{ld_I}(lN-(\bm r/\bm \alpha)_M)\bm\alpha$ is a subsolution of \eqref{Eq1}. Hence 
$$
\frac{1}{d_Il^m}(l^mN^m-(\bm r/\bm \alpha)_M)\bm\alpha<\bm U^{l^m},\quad \forall\ m\gg1.
$$
It follows from \eqref{Z-l=eq}-\eqref{Z-l-bounds} that 
\begin{align*}
\|\bm Z^{l^m}-\bm r\|_{1}=\Big\| \mathcal{L}\bm Z^{l^m}/(l^m\bm \beta\circ\bm U^{l^m})\Big\|_{1}
\le & \frac{d_I\|\mathcal{L}\|_1\|\bm Z^{l^m}\|_1}{\bm\beta_m(l^mN^m-(\bm r/\bm \alpha)_M)\bm\alpha_m}\cr 
\le &  \frac{d_I\|\mathcal{L}\|_1\bm r_M\|\bm \alpha\|_1}{\bm\beta_m(l^mN^m-(\bm r/\bm \alpha)_M)\bm\alpha_m^2} \to 0\quad \text{as}\ m\to\infty.
\end{align*}
Thus, 
$$ 
N^{*}_{d_I}=\lim_{m\to\infty}\sum_{j\in\Omega}Z_j^{l^m}=\sum_{j\in\Omega}r_j=\sum_{j\in\Omega}\frac{r_j}{\alpha_j}\alpha_j<(\bm r/\bm \alpha)_M\sum_{j\in\Omega}\alpha_j=(\bm r/\bm \alpha)_M.
$$

\noindent{\bf Case 2.} $L\in[R_*,\infty)$. Set $\bm W^{l^m}:=l^{m}{\bm U}^{l^m}$ for each $m\ge 1$. Then, by \eqref{Eq1},
\begin{equation}\label{Eq1-prime}
0=d_I\mathcal{L}\bm W^{l^m}+\big(\bm\beta\circ(l^mN^m\bm\alpha-d_I\bm W^{l^m})-\bm\gamma)\circ\bm W^{l^m},\quad \forall\ m>1.
\end{equation}
It is easy to check that $\frac{l^mN^m}{d_I}\bm\alpha$ is a super solution, so by the comparison principle we have $\bm W^{l^m}\le \frac{l^mN^m}{d_I}\bm\alpha$ for all $m\ge 1.$ Since $l^mN^m\to L$ as $m\to \infty$,   passing to a subsequence if necessary, we may suppose that $\bm W^{l^m}\to \bm W^{\infty}$ as $m\to\infty$. Note that if $L=R_*$, then $\bm W^{\infty}=\bm 0$, while $\bm W^{\infty}\gg \bm 0$ if $L>R_*$. Now, if $L=R_*$ then 
$$
\bm Z^{l^m}=l^mN^m\bm\alpha-d_I\bm W^{l^m}\to R_*\bm\alpha\quad \text{as}\ m\to\infty.
$$
This together with \eqref{DF3} implies that 
$$
N^*_{d_I}=\lim_{m\to\infty}\sum_{j\in\Omega}Z^{l^m}_j=R_*<(\bm r/\bm\alpha)_M.
$$
On the other had, if $L>R_*$, then $\bm W^{l^m}\to \bm W^{\infty}\gg \bm 0$ and  $\bm Z^{l^m}\to \bm Z^{\infty}:= L\bm \alpha-d_I\bm W^{\infty}$ as $m\to\infty.$ Moreover, thanks to \eqref{Eq1-prime},
$$
0=d_I\mathcal{L}\bm W^{\infty}+(\bm\beta\circ(L\bm \alpha-d_I\bm W^{\infty})-\bm\gamma)\circ \bm W^{\infty}.
$$
Since $L>R_*$, taking $N=L$ and  $l=1$ in \eqref{Eq1}, we have that $\bm W^{\infty}=\bm U^{1}$ and $\bm Z^{\infty}=\bm Z^{1}$. Therefore, by \eqref{DFF2},
$$
N_{d_I}^*=\lim_{m\to\infty}\mathcal{N}(l^m,N^m)=\sum_{j\in\Omega}Z^{\infty}_j=\mathcal{N}(1,L)<(\bm r/\bm \alpha)_M.
$$
Combing cases 1 and 2, we obtain that $N^*_{d_I}<(\bm r/\bm\alpha)_M$, which completes the proof of \eqref{N-critic-bounds}.

Fix $ N>N^*_{d_I}$.  Hence, $\mathcal{R}_0={N}/{R_*}>1$.  Thus by Theorem \ref{T2-mass-action-EE-Existence}, \eqref{model-mass-action} has an  EE solution $(\bm S(\cdot,d_S),\bm I(\cdot,d_S))$ for all $d_S>0$.  For $d_S>0$, notice that   $(\bm S(\cdot, d_S),\bm I(\cdot,d_S))=(\bm Z^l,d_Sl\bm U^{l})$, where $l={\kappa^*}/{d_S}$ is depending on $d_S$ and $\kappa^*$ is  as in Lemma \ref{lem3}-{\rm (i)}. Since
$$
N=\mathcal{N}(l, N)+ld_S\sum_{j\in\Omega}U^{l}_j,
$$
we have 
$$
l=\frac{N-\mathcal{N}(l, N)}{d_S\sum_{j\in\Omega}U^l_j}\ge \frac{N-N^*_{d_I}}{d_S\sum_{j\in\Omega}U^l_j}\to \infty\quad \text{as}\ d_S\to0.
$$
We can then proceed by the similar arguments after \eqref{RT32} to establish that $({\bm S}(\cdot,d_S),{\bm I}(\cdot,d_S))\to ({\bm r},(N-\sum_{j\in\Omega}r_j)\bm\alpha)$ as $d_S\to0$.
\end{proof}
\begin{rk} Suppose that the hypotheses of Proposition \eqref{thm-ee-ds-2} hold. Let $N^*_{d_I}$ be as in Proposition \eqref{thm-ee-ds-2}. Then alternative-{\rm (iii)} of Theorem \ref{thm-ee-ds} holds for any $N>N^*_{d_I}$, in particular for $N^*_{d_I}<N<(\bm r/\bm\alpha)_M$. If the  matrix $\mathcal{L}=(L_{i,j})_{ij=1}^n$ is symmetric, then $\bm \alpha =\frac{1}{n}{\bm 1}$ and $(\bm r/\bm \alpha)_M=n\bm r_{M}$. So $N^*_{d_I}<n\bm r_M$, and Proposition \ref{thm-ee-ds-2}  improves \cite[Theorem 7-{\rm (iii)}]{li2023sis} even for the symmetric case.   
\end{rk}

\subsection{The case of $d_I\to 0$}
Then we consider the asymptotic profile of the EE solutions of \eqref{model-mass-action} as $d_I$ approaches zero. For convenience, define 
\begin{equation} \label{Omega-star}
\Omega^*:=\left\{i\in\Omega\ :\ {r_i}/{\alpha_i}=(\bm r/\bm \alpha)_m\right\},
\end{equation}
and define
${ L}^{\Omega^*}_{ij}={ L}_{ij} 
 $ for each $i,j\in \Omega^*$. Given a nonempty subset $\mathcal{O}\subset\Omega$ and a column vector $\bm X\in\mathbb{R}^n$, we let $\bm X^{\mathcal{O}}=(X_i)^T_{i\in \mathcal{O}}$ as the projection of $\bm X$ on the Euclidean space $\mathbb{R}^{|\mathcal{O}|}$.

\begin{tm}\label{thm-ee-di}   Suppose that {\bf (A1)}-{\bf (A3)} holds, $\bm r\notin{\rm span}(\bm \alpha)$, $N>(\bm r/\bm\alpha)_m$ and  $d_S>0$ is fixed. Then there is $d_I^0>0$ such that \eqref{model-mass-action} has a unique EE solution for every $0<d_I<d_I^0$. The  EE solution $(\bm S,\bm I)$ of \eqref{model-mass-action} for $0<d_I\ll 1$ satisfies 
\begin{equation}\label{eq-1-ee-di}
    \lim_{d_I\to 0}\|\bm S-(\bm r/\bm\alpha)_m\bm\alpha\|=0, \quad  \lim_{d_I\to0}\sum_{j\in\Omega}I_i=N-(\bm r/\bm \alpha)_m, \quad \text{and}\quad \lim_{d_I\to0}I_i=0\quad \forall\ i\in\Omega\setminus\Omega^*.
\end{equation}
Moreover, restricting to a subsequence if necessary, $I_i\to I^*_i$ as $d_I\to0$ for each $i\in\Omega^*$, 
where $\bm I^*\in \mathbb{R}_+^{|\Omega^*|}$ is a solution of 
\begin{equation}\label{eq-2-ee-di}
    \begin{cases}
        0=d_S\mathcal{L}^{\Omega^*}\bm I^{*}+\bm\beta^{\Omega^*}\circ\left( C^*_{d_S}\bm\alpha^{\Omega^*}-\bm I^{*}\right)\circ\bm I^{*},\cr 
        \sum_{i\in\Omega^*}I^*_i=N-\big(\bm r/\bm \alpha\big)_m,
    \end{cases}
\end{equation}
 where the constant $C^*_{d_S}$ satisfies
\begin{equation}\label{eq-3-ee-di}
\frac{N-\big(\bm r/\bm\alpha\big)_m}{\sum_{j\in\Omega^*}\alpha_j}\le C^*_{d_S}\le \frac{(N-(\bm r/\bm\alpha)_m)\bm\beta_m+d_S\|\mathcal{L}^{\Omega^*}\|_{1}}{\alpha_m\bm\beta_m}.
\end{equation}  
Furthermore,   
\begin{equation}\label{eq-4-ee-di}
   \lim_{d_S\to 0} \big\|\bm I^*- \frac{N-(\bm r/\bm \alpha)_m}{\sum_{j\in{\Omega}^*}\alpha_j}\bm\alpha^{{\Omega}^*}\big\|_1=0.
\end{equation}    
\end{tm}
\begin{proof} {First, note that the assumption  $\bm r\notin{\rm span}(\bm \alpha)$ implies that the sets $\Omega^*$ and $\Omega\setminus\Omega^*$ are both nonempty.}  By \cite[Theorem 2.7]{chen2020asymptotic}, it holds that 
\begin{equation}\label{R-0-di-zero-limit}
\lim_{d_I\to 0}\mathcal{R}_0=N(\bm \beta\circ \bm \alpha/\bm \gamma)_M=\frac{N}{(\bm \alpha/\bm r)_m}.
\end{equation}
Since $N>(\bm\alpha/\bm r)_m$,  there is $d_I^1>0$ such that $\mathcal{R}_0>1$ for all $d_I\in (0, d_I^1)$. By Theorem \ref{T2-mass-action-EE-Existence}-{\rm (ii)}, \eqref{model-mass-action} has a unique EE solution for all $d_I\in (0, d_I^0)$, where $d_I^0:=\min\{d_I^1,d_S\}$. 

Let $(\bm S(\cdot, d_I),\bm I(\cdot, d_I))$ be the EE solution of \eqref{model-mass-action} for $0<d_I\le d_I^0$. For each $0<d_I\le d_I^0$, let $\kappa^*_{d_I}$ be as in Lemma \ref{lem3}-{\rm (i)}. Then 
\begin{equation}\label{RT70}
\kappa^*_{d_I}N\bm\alpha=d_S\bm S(\cdot,d_I)+d_I\bm I(\cdot,d_I),\quad \forall\ 0<d_I\le d_I^0.
\end{equation}
Since $d_I\bm I(\cdot,d_I)\to \bm 0$ as $d_I\to 0$ and $(\bm r/\bm \alpha)_m\bm\alpha\le \bm S(\cdot,d_I)\le (\bm r/\bm \alpha)_M\bm\alpha $  for any $0<d_I\le d_I^0$, passing to a subsequence if necessary,  we may suppose that 
$$
\kappa^*_{d_I}\to \kappa^*_0\in\left[\frac{d_S(\bm r/\bm \alpha)_m}{ N}, \frac{d_S(\bm r/\bm \alpha)_M}{N}\right] \quad\text{as}\quad d_I\to 0.
$$ 
By \eqref{RT70}, $\bm S(\cdot, d_I)\to \frac{\kappa^*_0{N}}{d_S}\bm\alpha$ as $d_I\to0$.  It then follows from \cite[Lemma 2.5-{\rm (ii)}]{chen2020asymptotic} that 
\begin{equation*}
    \lim_{d_I\to0}\sigma_*(d_I\mathcal{L}+{\rm diag}(\bm\beta\circ \bm S(\cdot, d_I)-\bm \gamma))=\big(\frac{\kappa^*_0{N}}{d_S}\bm \beta\circ\bm \alpha-\bm \gamma\big)_M.
\end{equation*}
Since $ \sigma_*(d_I\mathcal{L}+{\rm diag}(\bm\beta\circ \bm S(\cdot, d_I)-\bm \gamma))=0$ for all $0<d_I\le d_I^0$, we obtain that  $(\frac{\kappa^*_0{N}}{d_S}\bm \beta\circ\bm \alpha-\bm \gamma)_M=0$, i.e., 
$ {\kappa_0^*}=d_S(\bm r/\bm \alpha)_m/N.$ Since $\kappa^*_0$ is independent of the subsequence, we conclude that $\kappa^*_{d_I}\to  {d_S}(\bm r/\bm \alpha)_m/N$ and $\bm S(\cdot,d_I)\to (\bm r/\bm \alpha)_m\bm\alpha$ as $d_I\to 0$. From which we  have that $\sum_{i\in\Omega}I_i(\cdot, d_I)=N-\sum_{i\in\Omega}S_i(\cdot,d_I)\to N-(\bm r/\bm \alpha)_m$ as $d_I\to 0$. Noticing $d_I\mathcal{L}\bm I+\bm\beta\circ(\bm S(\cdot, d_I)-\bm r)\circ \bm I(\cdot, d_I)=\bm 0$, we have
$$
\lim_{d_I\to0}(\bm S(\cdot, d_I)-\bm r)\circ \bm I(\cdot, d_I)=\bm 0,
$$
which implies that $I_i(\cdot, d_I)\to 0$ as $d_I\to 0$ for every $i\in\Omega\setminus\Omega^*$. This proves \eqref{eq-1-ee-di}.  

Next, we determine the limit of $I_i(\cdot, d_I)$ as $d_I\to 0$ for $i\in\Omega^*$.
To this end, we set  $F^{\Omega^*}_i(\cdot,d_I)=\sum_{j\in\Omega\setminus\Omega^*}L_{ij}I_j(\cdot,d_I)$ for each $i\in\Omega^*$. By the equation of $\bm I$
and \eqref{RT70}, 
\begin{equation}\label{RT71}
    0=d_S\mathcal{L}^{\Omega^*}\bm I^{\Omega^*}(\cdot, d_I)+\bm\beta^{\Omega^*}\circ\left( \frac{d_S\Big(\frac{\kappa^*_{d_I}}{d_S}N-\big(\bm r/\bm\alpha\big)_m\Big)}{d_I}\bm \alpha^{\Omega^*}-\bm I^{\Omega^*}(\cdot,d_I)\right)\circ\bm I^{\Omega^*}(\cdot,d_I)+d_S\bm F^{\Omega^*}(\cdot, d_I).
\end{equation}
    Note that $ F_i^{\Omega^*}\to 0 $ as $d_I\to 0$ for each $i\in\Omega^*$.  By \eqref{RT70} and the fact that $\big(\bm r/\bm \alpha\big)_m\bm\alpha<\bm S$, we obtain that 
    $$
    \frac{d_S\Big(\frac{\kappa^*_{d_I}}{d_S}N-\big(\bm r/\bm\alpha\big)_m\Big)}{d_I}\bm \alpha^{\Omega^*}>\frac{d_S}{d_I}\Big(\frac{\kappa^*_{d_I}}{d_S}N\bm \alpha^{\Omega^*}-\bm S(\cdot,d_I)\Big)= \bm I^{\Omega^*},
    $$
  which implies that
    $$
    \frac{d_S\Big(\frac{\kappa^*_{d_I}}{d_S}N-\big(\bm r/\bm\alpha\big)_m\Big)}{d_I}>\frac{\sum_{j\in\Omega^*}I_j}{\sum_{j\in\Omega^*}\alpha_j}.
    $$
Since $\lim_{d_I\to0}\sum_{j\in\Omega^*}I_j(\cdot,d_I)=N-\big(\bm r/\bm \alpha\big)_m$,
    \begin{equation}\label{RT72}
        \liminf_{d_I\to 0}\frac{d_S\Big(\frac{\kappa^*_{d_I}}{d_S}N-\big(\bm r/\bm\alpha\big)_m\Big)}{d_I}\ge\eta_1:=\frac{N-\big(\bm r/\bm\alpha\big)_m}{\sum_{j\in\Omega^*}\alpha_j} .
    \end{equation}
  We rewrite  \eqref{RT71} as
    $$
   \frac{d_S\Big(\frac{\kappa^*_{d_I}}{d_S}N-\big(\bm r/\bm\alpha\big)_m\Big)}{d_I}\bm\alpha^{\Omega^*}\circ\bm I^{\Omega^*}= \bm I^{\Omega^*}(\cdot,d_I)\circ\bm I^{\Omega^*}(\cdot,d_I)-d_S(\mathcal{L}^{\Omega^*}\bm I^{\Omega^*}(\cdot,d_I)+\bm F^{\Omega^*}(\cdot,d_I))/\bm \beta^{\Omega^*}.
    $$
Therefore, 
    $$
\sum_{j\in\Omega^*}\Big((I_j^{\Omega^*})^2+\frac{d_S}{\bm\beta_m}\big(\|\mathcal{L}^{\Omega^*}\|_{1}\|\bm I^{\Omega^*}\|_1+F_j^{\Omega^*}\big)\Big)\ge \frac{d_S\Big(\frac{\kappa^*_{d_I}}{d_S}N-\big(\bm r/\bm\alpha\big)_m\Big)}{d_I}\sum_{j\in\Omega^*}\alpha_j I_j.
    $$
Since $\sum_{j\in\Omega^*}I_j^{\Omega^*}\to N-(\bm r/\bm\alpha)_m$ and $F_i\to 0$ as $d_I\to 0$ for each $i\in\Omega^*$, we have that
    \begin{align}\label{RT73}    \limsup_{d_I\to0}\frac{d_S\Big(\frac{\kappa^*_{d_I}}{d_S}N-\big(\bm r/\bm\alpha\big)_m\Big)}{d_I}\le& \frac{(N-(\bm r/\bm\alpha)_m)^2+\frac{d_S\|\mathcal{L}^{\Omega^*}\|_1}{\bm \beta_m}(N-(\bm r/\bm\alpha)_m) }{\bm\alpha_m(N-(\bm r/\bm\alpha)_m)}\cr
        =&\frac{(N-(\bm r/\bm\alpha)_m)\bm\beta_m+d_S\|\mathcal{L}^{\Omega^*}\|_1}{\bm\alpha_m\bm\beta_m}:=\eta_2.
    \end{align}
    Therefore, possible after passing to a subsequence, it follows from \eqref{RT71}, \eqref{RT72} and \eqref{RT73} that $ \bm I^{\Omega^*}\to \bm I^*$ as $d_I\to 0$, where $\bm I^*$ solves \eqref{eq-2-ee-di} and  $ C^*_{d_S}\in[\eta_1,\eta_2]$. 
    

    Finally, we prove \eqref{eq-4-ee-di}. To see this, we rewrite  \eqref{RT71} as
    $$
   \bm 0= \bm I^{\Omega^*}(\cdot,d_I)\circ\bm I^{\Omega^*}(\cdot,d_I)-\frac{d_S\Big(\frac{\kappa^*_{d_I}}{d_S}N-\big(\bm r/\bm\alpha\big)_m\Big)}{d_I}\bm\alpha^{\Omega^*}\circ\bm I^{\Omega^*}-d_S(\mathcal{L}^{\Omega^*}\bm I^{\Omega^*}(\cdot,d_I)+\bm F^{\Omega^*}(\cdot,d_I))/\bm \beta^{\Omega^*},
    $$
    and set  
    $$
    Q_i:=\frac{d_S\Big(\frac{\kappa^*_{d_I}}{d_S}N-\big(\bm r/\bm\alpha\big)_m\Big)}{d_I}\alpha_i^{\Omega^*}-\frac{d_S}{\beta_i}L_{ii},\ \ \forall\, i\in\Omega^*.
    $$ 
By the quadratic formula,  we have   
\begin{equation*}
    I_i^{\Omega^*}= \frac{1}{2}\left[ Q_i+\sqrt{Q_i^2+4\frac{d_S}{\beta_i}\Big(\sum_{j\in\Omega^*\setminus\{i\}}L_{ij}I_i^{\Omega^*}+F_i^{\Omega^*}\Big)}\right]\ge \frac{1}{2}(Q_i+|Q_i|),\quad \forall\ i\in\Omega^*.
\end{equation*}
Hence for each $i\in\Omega^*,$
$$
I^*_i=\lim_{d_I\to 0}I_i^{\Omega^*}\ge \lim_{d_I\to0}\frac{1}{2}(Q_i+|Q_i|)=\frac{1}{2}\Big(C^*_{d_S}\alpha_i-\frac{d_S}{\beta_i}L_{ii}+\big|C^*_{d_S}\alpha_i-\frac{d_S}{\beta_i}L_{ii}\big|\Big).
$$
 Since $\inf_{d_S>0}C^*_{d_S}\ge \eta_1>0$,  we can choose $0<d_S^{0}\ll 1$ such that 
\begin{equation}\label{RT73-3}
    I_i^*\ge \frac{\eta_1}{2}\alpha_{i},\quad \forall\ i\in\Omega^*,\ 0<d_S<d_S^0.
\end{equation}
By $d_S\mathcal{L}^{\Omega^*}\bm I^*\to 0$ as $d_S\to 0$ and \eqref{eq-3-ee-di}, passing to a subsequence if necessary, it follows from \eqref{RT73} and the first equation in \eqref{eq-2-ee-di} that there is a positive number  $C^*_0$ such that $\bm I^*\to C^*_0\bm \alpha^{\Omega^*}$. Thanks to the second equation in \eqref{eq-2-ee-di}, we get that $C^*_0=(N-(\bm r/\bm\alpha)_m)/\sum_{j\in\Omega^*}\alpha_j$. This shows that $C^*_0$ is independent of the chosen subsequence. Therefore,  \eqref{eq-4-ee-di} holds.
\end{proof}

\begin{rk}\begin{itemize}
\item[\rm (i)]Suppose that $N>\big(\bm r/\bm\alpha\big)_m$ and $\mathcal{L}^{\Omega^*}$ is irreducible. Then the positive constant $C^*_{d_S}$ and the vector $ \bm I^*$ solving \eqref{eq-2-ee-di} are uniquely determined. Moreover, $I^*_i>0$ for each $i\in\Omega^*$ and $\bm I^{\Omega^*}\to \bm I^*$ as $d_I\to 0$. If in addition $\bm \alpha^{\Omega^*} $ is an eigenvector associated to $\sigma_*(\mathcal{L}^{\Omega^*})$, we have  $C_{d_S}^*=(N-\big(\bm r/\bm\alpha\big)_m)/(\sum_{j\in\Omega^*}\alpha_j)$ that is independent of $d_S$, and $\bm I^*=C_{d_S}^*\bm\alpha^{\Omega^*}$.

\item[\rm (ii)] When the dispersal rate of the susceptible population is small and the dispersal rate of the infected population is even smaller, \eqref{eq-4-ee-di} implies that the infected population will concentrate exactly  on all the most high-risk patches. This result seems to be new even for the case when    $\mathcal{L}$ is symmetric. 

\item[\rm (iii)] Theorem \ref{thm-ee-di} extends and improves \cite[Theorem 8]{li2023sis}, which is for the case when $\mathcal{L}$ is symmetric. 
\end{itemize}
    
\end{rk}

\subsection{The case of $d_I\to 0$ and $d_I/d_S\to \sigma\in [0, \infty]$}
We complete our investigation with the study of the profiles of the EE when both dispersal rates are small.

\begin{tm}\label{T6-ee-di-ds-small}  Suppose that {\bf (A1)}-{\bf (A3)} holds, $\bm r\notin{\rm span}(\bm \alpha) $, $N>(\bm r/\bm\alpha)_m$ and  $\sigma>0$ is fixed. The EE solution $(\bm S,\bm I)$ of \eqref{model-mass-action} for $0<d_I\ll 1$ and $|\frac{d_I}{d_S}-\sigma|\ll 1$ satisfies
\begin{equation}\label{Eq1-T6-ee-di-ds-small}
    \lim_{d_I+|\frac{d_I}{d_S}-\sigma|\to 0}\big\|\bm S-\bm S^{\sigma}\big\|_1=0\quad \text{and}\quad \lim_{d_I+|\frac{d_I}{d_S}-\sigma|\to 0}\|\bm I-\bm I^{\sigma}\|_1=0,
\end{equation}
where $( S^{\sigma}_i, I^{\sigma}_i):=(\min\{l^{\sigma}N\alpha_i, r_i\},\frac{1}{\sigma}(l^{\sigma}N\alpha_i- r_i)_+)$ for each $i\in\Omega$, and $l^{\sigma}$ is the unique positive solution of the algebraic equation
\begin{equation}\label{Eq2-T6-ee-di-ds-small}
    N=\sum_{j\in\Omega}\big[\min\{l^{\sigma}N\bm\alpha_j,\bm r_j\}+\frac{1}{\sigma}(l^{\sigma}N\bm\alpha_j-\bm r_j)_+\big].
\end{equation}
Furthermore, the following conclusions on $(\bm S^{\sigma},\bm I^{\sigma})$ and $l^{\sigma}$ hold.
\begin{itemize}
    \item[\rm (i)] As $\sigma\to 0$, then $l^{\sigma}\to\frac{(\bm r/\bm\alpha)_m}{N}$, $S^{\sigma}\to (\bm r/\bm\alpha)_m\bm\alpha$, $I_j^{\sigma}\to 0$ for  $j\in\Omega\setminus\Omega^*$, and $I_j^{\sigma}\to \frac{N-(\bm r/\bm\alpha )_m}{\sum_{i\in\Omega^*}\alpha_i}\alpha_j$ for  $j\in\Omega^*$, where $\Omega^*$ is defined as in \eqref{Omega-star}.
    \item[\rm (ii)] As $\sigma\to\infty$, the following statements hold.
    \begin{itemize}
        \item[\rm (ii-1)] If $N<\sum_{j\in\Omega}r_j$, then $l^{\sigma}\to l^{\infty}$, where $l^{\infty}\in\Big(\frac{(\bm r/\bm\alpha)_m}{N},\frac{(\bm r/\bm\alpha)_M}{N}\Big)$ satisfies  the algebraic equation $N=\sum_{j\in\Omega}\min\{l^{\infty}N\alpha_j,r_j\}$, $\bm S^{\sigma}\to \min\{l^{\infty}N\bm\alpha,\bm r\} $, and $\bm I^{\sigma}\to \bm 0.$ 
        
        \item[\rm (ii-2)] If $N=\sum_{j\in\Omega}r_j$, then $l^{\sigma}\to \frac{(\bm r/\bm\alpha)_M}{N}$, $\bm S^{\sigma}\to \bm r$, and $\bm I^{\sigma}\to \bm 0$. 
        
        \item[\rm (ii-3)] If $N>\sum_{j\in\Omega}r_j$, then  $l^{\sigma}=\frac{\sigma(N-\sum_{j\in\Omega}r_j)+\sum_{j\in\Omega}r_j }{N}$ for $\sigma\gg 1$,  $\bm S^{\sigma}\to \bm r$ and $ \bm I^{\sigma}\to (N-\sum_{j\in\Omega}r_j)\bm\alpha$.
    \end{itemize}
\end{itemize}
    
\end{tm}
\begin{proof} First, it is easy to see that the function $G_{\sigma}: [0,\infty)\to [0,\infty) $ defined by 
\begin{equation*}
    G_{\sigma}(l)=\sum_{j\in\Omega}\Big[\min\{lN\alpha_j,r_j\}+\frac{1}{\sigma}(lN\alpha_j-r_j)_+\Big],\quad \forall\ l\ge 0
\end{equation*}
is continuous,  strictly increasing, and satisfies 

\begin{equation*}
    G_{\sigma}(0)=0 \quad \text{and}\quad \lim_{l\to\infty}G_{\sigma}(l)=\infty.
\end{equation*}
Thus there is a unique $l^{\sigma}>0$ such that $G_{\sigma}(l^{\sigma})=N$.  Note also that 
\begin{equation}\label{DFF2-1}
    G_{\sigma}(l)=\begin{cases}
       lN, & \text{if}\quad 0\le l\le \frac{(\bm r/\bm\alpha)_m}{N},\cr
       \sum_{j\in\Omega}r_j +\frac{1}{\sigma}(lN-\sum_{j\in\Omega}r_j), & \text{if}\ l\ge \frac{(\bm r/\bm \alpha )_M}{N}.
    \end{cases}
\end{equation}

By the assumption $N>(\bm r/\bm \alpha)_m$ and \eqref{R-0-di-zero-limit}, there is $d_I^0>0$ such that $\mathcal{R}_0>1$ for all $0<d_I<d_I^0$. Hence by Theorem \ref{T2-mass-action-EE-Existence},  \eqref{model-mass-action} has at least one EE solution  $(\bm S,\bm I)$  for all $d_S>0$ and $0<d_I<d_I^0$. By Lemma \ref{lem3}, there is $l^{d_I,d_S}>{1}/{\mathcal{R}_0}$ such that $(\bm S,\bm I)=(l^{d_I,d_S}(N\bm\alpha-d_I\bm U^{l^{d_I,d_s}}),d_Sl^{d_I,d_S}\bm U^{l})$. Multiplying \eqref{Eq1} by $d_Sl^{d_I,d_S}$, we have 
\begin{equation}\label{Y1}
    \bm 0=d_I\mathcal{L}\bm I +\bm\beta\circ\big((l^{d_I,d_S}N\bm\alpha-\bm r)-\frac{d_I}{d_S}\bm I\big)\circ\bm I.
\end{equation}

It is easy to check that $\underline{\bm I}:=\frac{d_S}{d_I}(l^{d_I,d_S}N-(\bm r/\bm \alpha)_M)_+\bm\alpha$ is a subsolution of \eqref{Y1}, then 
$$(l^{d_I,d_S}N-(\bm r/\bm \alpha)_M)_+\bm\alpha  \le\frac{d_I}{d_S}\bm I, $$
    which implies that 
    $$
    (l^{d_I,d_S}N-(\bm r/\bm \alpha)_M)_+\le \frac{d_I}{d_S}\sum_{j\in\Omega}I_j\le\frac{d_I}{d_S} N.
    $$
    Hence,
    \begin{equation*}
       \frac{1}{\mathcal{R}_0}< l^{d_I,d_S}\le \frac{(\bm r/\bm\alpha)_M}{N}+\frac{d_I}{d_S}.
    \end{equation*}
    Therefore, possibly after passing to a subsequence, we may suppose that there is \begin{equation}\label{Y0}
    \frac{(\bm r/\bm\alpha)_m}{N}\le l^*\le\frac{(\bm r/\bm\alpha)_M}{N}+\sigma 
    \end{equation}
    such that  $ l^{d_I,d_S}\to l^{*}$ as $d_I+|\frac{d_I}{d_S}-\sigma|\to 0$. Observe from \eqref{Y1} that 
    $$
    0=d_I\sum_{j\in\Omega}L_{ij}I_j +\big(\beta_i(l^{d_I,d_S}N\alpha_i-r_i)-d_I\sum_{j\in\Omega}L_{ji}\big)I_i -\frac{d_I}{d_S}\beta_iI_i^2,\quad i\in\Omega.
    $$
   By the quadratic formula,
    \begin{small}
    \begin{equation}\label{Y2}
        I_i=\frac{\sqrt{\big(\beta_i(l^{d_I,d_S}N\alpha_i-r_i)-d_I\sum_{j\in\Omega}L_{ji}\big)^2+4\beta_i\frac{d_I}{d_S}d_I\sum_{j\in\Omega}L_{ij}I_j}+\big(\beta_i(l^{d_I,d_S}N\alpha_i-r_i)-d_I\sum_{j\in\Omega}L_{ji}\big)}{2\beta_i\frac{d_I}{d_S}},\quad i\in\Omega.
    \end{equation}
    \end{small}
    Letting $d_I+|\frac{d_I}{d_S}-\sigma|\to 0$ in\eqref{Y2}, we obtain that 
    \begin{equation}\label{Y3}
        \lim_{d_I+|\frac{d_I}{d_S}-\sigma|\to0}I_i=\frac{1}{\sigma}\big(l^*N\alpha_i-r_i\big)_+,\quad i\in\Omega.
    \end{equation}
    This in turn implies that 
    \begin{equation}\label{Y4}
        \lim_{d_I+|\frac{d_I}{d_S}-\sigma|\to0}S_i=\lim_{d_I+|\frac{d_I}{d_S}-\sigma|\to0}\big(l^{d_I,d_I}N\alpha_i-\frac{d_I}{d_S}I_i)=l^*N\alpha_i-(l^{*}N\alpha_i-r_i)_+=\min\{l^*N\alpha_i,r_i\},\quad i\in\Omega.
    \end{equation}
    Since $N=\sum_{j\in\Omega}(S_i+I_i)$, we deduce from \eqref{Y3}-\eqref{Y4} that $G_{\sigma}(l^*)=N$. Thus, $l^*=l^{\sigma}$, which implies that $l^*$ is independent of the chosen subsequence. Therefore, $l^{d_I,d_S}\to l^{\sigma}$ as $d_I+|\frac{d_I}{d_S}-\sigma|\to 0$. It then follows from \eqref{Y3}-\eqref{Y4} that \eqref{Eq1-T6-ee-di-ds-small} holds.

    {\rm (i)} In view of \eqref{Y0}, restricting to a subsequence if necessary, we may assume $l^\sigma\to l^0$ for some $l^0\ge \frac{(\bm r/\bm\alpha)_m}{N}$ as $\sigma\to 0$.
    By \eqref{Eq2-T6-ee-di-ds-small}, $(l^{\sigma}N\alpha_j-r_{j})_+\le \sigma N$ for each $ j\in\Omega$. Taking $\sigma\to 0$, we have $(l^{0}N\alpha_j-r_{j})_+\le 0$
    for each $j\in\Omega$, which implies $l^0\le \frac{(\bm r/\bm\alpha)_m}{N}$. Therefore, $l^{\sigma}\to \frac{(\bm r/\bm\alpha)_m}{N}$ as $\sigma\to 0$. 
    By \eqref{Y4}, 
    $$
    S_i^{\sigma}\to \min\{(\bm r/\bm \alpha)_m\alpha_i,r_i\}=(\bm r/\bm\alpha)_m\alpha_i\quad \text{as} \quad \sigma\to0, \quad\forall\, i\in\Omega.
    $$
    This in turn implies that 
    $$
    \sum_{j\in\Omega}I_j=N-\sum_{j\in\Omega}S_j\to N-(\bm r/\bm \alpha)_m \quad\text{as}\quad \sigma\to 0.
    $$ 

    If $j\in\Omega\setminus\Omega^*$, then $(\bm r/\bm\alpha)_m\alpha_j<r_j$ and $I_i^{\sigma}\to 0$ as $\sigma\to 0$  by \eqref{Y3}. Next, observe that  
    \begin{equation}\label{Y5}
    I_j^{\sigma}=\frac{1}{\sigma}(l^{\sigma}N\alpha_j-r_j)_+=\frac{(l^{\sigma}N-(\bm r/\bm\alpha)_m)_+}{\sigma}\alpha_j,\quad \forall\ j\in\Omega^*.
    \end{equation}
    Hence,
    \begin{equation}\label{Y6}
\sum_{j\in\Omega^*}I_j^{\sigma}=\frac{(l^{\sigma}N-(\bm r/\bm\alpha)_m)_+}{\sigma}\sum_{j\in\Omega^*}\alpha_j.
    \end{equation}
    Since $\sum_{j\in\Omega\setminus\Omega^*}I_{j}^{\sigma}\to 0$ and $\sum_{j\in\Omega}I_j^{\sigma}\to N-(\bm r/\bm\alpha)_m$ as $\sigma\to 0$, we get from \eqref{Y6} that 
    $$
    \frac{(l^{\sigma}N-(\bm r/\bm\alpha)_m)_+}{\sigma}\to\frac{N-(\bm r/\bm\alpha)_m}{\sum_{j\in\Omega ^*}\alpha_j}\quad  \text{as}\quad \sigma\to0.
    $$ 
    We can invoke  \eqref{Y5} to conclude that 
    $$
    I_j^{\sigma}\to \frac{N-(\bm r/\bm\alpha )_m}{\sum_{i\in\Omega^*}\alpha_i}\alpha_j \quad \text{as} \quad\sigma\to 0,\quad \forall j\in\Omega^*.
    $$

    {\rm (ii)} We consider $\sigma\to \infty$ and show that {\rm (ii-1)}-{\rm (ii-3)} holds.

    {\rm (ii-1)} Suppose that $N<\sum_{j\in\Omega}r_j$. Then,  $$G_{\sigma}(l^{\sigma})=N<\sum_{j\in\Omega}r_j<G_{\sigma}\big( \frac{(\bm r/\bm\alpha)_M}{N}\big)$$
    and 
    $$
    G_{\sigma}\big(\frac{(\bm r/\bm\alpha)_m}{N}\big)= (\bm r/\bm\alpha)_m<N= G_{\sigma}(l^{\sigma}).
    $$ 
    It then follows from the monotonicity of $G_\sigma$ that 
    $$
    \frac{(\bm r/\bm\alpha)_m}{N}< l^{\sigma}<\frac{(\bm r/\bm\alpha)_M}{N}, 
    \quad\forall\;\sigma>0.
    $$ 
    Therefore,  
    $$
    I_i^{\sigma}=\frac{(l^{\sigma}N-r_i)_{+}}{\sigma}\alpha_i\to 0 
\quad \text{as}\quad \sigma\to \infty, \quad \forall\; i\in\Omega.
$$
Passing to a subsequence if necessary, there is $l^{\infty}\in[\frac{(\bm r/\bm\alpha)_m}{N},\frac{(\bm r/\bm\alpha)_M}{N}]$ such that $l^{\sigma}\to l^{\infty}$ and $\bm S^{\sigma}\to \min\{l^{\infty}N\bm \alpha,\bm r\}$ as $\sigma\to\infty$. Taking $\sigma\to \infty$ in $G_{\sigma}(l^{\sigma})=N$, we have 
    \begin{equation}\label{DFF2-2}
    N=\sum_{j\in\Omega}\min\{l^{\infty}N\alpha_i,r_j\}.
    \end{equation}
Since $N>(\bm r/\bm \alpha)_m$ and $\bm r\notin{\rm span}(\bm \alpha)$, 
 there is a unique $l^{\infty}\in \Big(\frac{(\bm r/\bm\alpha)_m}{N},\frac{(\bm r/\bm\alpha)_M}{N}\Big)$ satisfying \eqref{DFF2-2}.  So $l^{\infty}$ is independent of the chosen subsequence, which yields the desired result.

    {\rm (ii-2)} Suppose that $N=\sum_{j\in\Omega}r_j$. Similar to {\rm (ii-1)}, we can show that $l^{\sigma}\to l^{\infty}$, $\bm I^{\sigma}\to 0$, and $\bm S^{\sigma}\to \min\{l^{\infty}N\bm\alpha,\bm r\}$ as $\sigma\to\infty$, where $l^{\infty}$ is the unique positive solution of \eqref{DFF2-2} in $\Big[\frac{(\bm r/\bm \alpha)_m}{N},\frac{(\bm r/\bm\alpha)}{N}\Big]$. Now since $N=\sum_{j\in\Omega}r_j$,  we must have  $l^{\infty}=\frac{(\bm r/\bm\alpha)_M}{N}$, which yields the desired result.

    {\rm (ii-3)} Suppose that $N>\sum_{j\in\Omega}r_j$. Observing that
    $$
  N=  G_{\sigma}(l^{\sigma})\le \sum_{j\in\Omega}r_j+\frac{l^{\sigma}N}{\sigma}, \quad \forall\ \sigma>0,
    $$
   we have
    $$
    l^{\sigma}\ge \frac{\sigma(N-\sum_{j\in\Omega}r_j)}{N}\to \infty\quad \text{as}\quad \sigma\to\infty.
 $$    
    Thus by \eqref{DFF2-1},  
    $$
    N=\sum_{j\in\Omega}r_j+\frac{1}{\sigma}(l^{\sigma}N-\sum_{j\in\Omega}r_j),\quad \forall\,\sigma\gg 1.
    $$ 
    Hence, 
    $$
    l^{\sigma}=\frac{\sigma (N-\sum_{j\in\Omega}r_j)+\sum_{j\in\Omega}r_j}{N}, \quad\forall \, \sigma\gg 1.
    $$
    As a result, for each $j\in\Omega$, 
    $$
     I^{\sigma}_j=\frac{1}{\sigma}\Big( \Big(\sigma(N-\sum_{i\in\Omega}r_i)+\sum_{i\in\Omega}r_i\Big)\alpha_j-r_{j}\Big)_+
    \to \Big(N-\sum_{i\in\Omega}r_i\Big)\alpha_j\quad\text{as}\quad \sigma\to\infty$$ and 
    $$
    S_j^{\sigma}=\min\Big\{\Big(\sigma\Big(N-\sum_{i\in\Omega}r_i\Big)+\sum_{i\in\Omega}r_i\Big)\alpha_j,r_j\Big\}\to r_j\quad \text{as}\quad \sigma\to\infty.
    $$

\end{proof}

\section{Discussions}

We studied the dynamics of solutions of an  epidemic model \eqref{model-mass-action} with mass-action transmission mechanism in discrete-space environment. We defined the basic reproduction number $\mathcal{R}_0$ (formula \eqref{R-0}) and  studied the threshold dynamics of the model in terms of $\mathcal{R}_0$. In particular, when $\mathcal{R}_0\le 1$, we established results on the global stability of the DFE (Theorem \ref{T1-mass-action-DEF-stability}), and the existence and multiplicity of EE solutions (Theorems \ref{T2-mass-action-EE-Multiple}-\ref{theorem-ds}). When $\mathcal{R}_0>1$, we obtained results on the existence and uniqueness of EE solution (Theorem \ref{T2-mass-action-EE-Existence}). The global stability of the EE solution is  established under some additional hypotheses (Theorem \ref{T1-mass-action-EE-stability}).  In the second part of the manuscript, we  studied the impact of  restricting population movement on the dynamics of the disease by investigating the asymptotic profiles of the EE solutions as $d_S$ or/and $d_I$ approach zero (Theorems \ref{thm-ee-ds}, \ref{thm-ee-di}, and \ref{T6-ee-di-ds-small}). This study may help in taking informed decisions to implementing effective disease control strategies. 

Our results are novel in two aspects. On one hand, we drop the assumption that the connectivity matrix between the patches is symmetric. As a consequence, some of our results generalize  previous  results obtained by Li and Peng \cite{li2023sis}. On the other hand, some of our results are new even when the connectivity matrix is symmetric.
Firstly, the results on the global stability of DFE in Theorem \ref{T1-mass-action-DEF-stability}-{\rm (i)}\& {\rm (ii)} seem to be new. Secondly, the results on the existence of multiple EE solutions  are new. In simple epidemic models such as model \eqref{model-mass-action} with one patch, it is usually expected that the DFE is globally stable if $\mathcal{R}_0<1$.  Our Theorem \ref{T2-mass-action-EE-Multiple} indicates  that model \eqref{model-mass-action} with more than one patch may have multiple EE solutions and  the disease can persist  even if $\mathcal{R}_0<1$. This result highlights a major difference on the predictions of the disease dynamics from models using the standard transmission mechanism:  in model \eqref{model-si}, the disease always extinct if $\mathcal{R}_0<1$ \cite{allen2007asymptotic,chen2020asymptotic}. Another implication of this result is that the disease dynamics can be more complicated in patchy environment.

Our results are also interesting when using the dispersal rate $d_S$ or total population $N$ as bifurcation parameters. Notice that $\mathcal{R}_0$ is independent of $d_S$.  If $\mathcal{R}_0<1$, Theorem \ref{T1-mass-action-DEF-stability}-{\rm (ii)} states that the DFE is globally asymptotically stable if $d_S$ is sufficiently large; Theorem \ref{T2-mass-action-EE-Multiple} indicates  that it may be possible for the disease to persist  if $d_S$ is sufficiently small. If  $\mathcal{R}_0>1$,  there is some critical number $d_I^*=(1-m_{d_I})d_I$, where  $m_{d_I}\in (0, 1]$  depends on the dispersal rate $d_I$, such that the (unique) EE solution exists if and only if  $\mathcal{R}_0>1$  (Theorem \ref{T2-mass-action-EE-Existence}). Interestingly, $d_I^*$ is independent of $N$  and equals zero if $\bm r\in {\rm span}(\bm \alpha)$.  Moreover, by Theorem \ref{theorem-ds} and Remark \ref{R-1}, we have  $d^*_I>0$ if $\sum_{j\in\Omega}\gamma_j/\beta_j<(\sum_{j\in\Omega}\alpha_j\gamma_j)/(\sum_{j\in\Omega}\beta_i\alpha_j^2)$ and $d_I$ is  large.  It appears to be a difficult problem to understand the dependence of $d_I^*$ on $d_I$. It is worth pointing out that when the connectivity matrix $\mathcal{L}$ is symmetric, \cite[Theorem 4-{\rm (i)} $\&$ {\rm (iii)} ]{li2023sis} established the existence and uniqueness of the EE when $\mathcal{R}_0>1$ and $d_S>(1-|\Omega|/(\sum_{j\in\Omega}\beta_j)^{1/3}(\sum_{j\in\Omega}\beta_j^{-1/2})^{2/3})d_I$. Hence, when $\mathcal{L}$ is symmetric we have an apriori upper bound for $d_I^*$:
$$
d_I^*\le\Big(1-|\Omega|/(\sum_{j\in\Omega}\beta_j)^{1/3}(\sum_{j\in\Omega}\beta_j^{-1/2})^{2/3}\Big)d_I.
$$
We also note that $\mathcal{R}_0$ is strictly decreasing in total population $N$. If $N$ is sufficiently small ($\mathcal{R}_0\le 1$ in this case),  Theorem \ref{T1-mass-action-DEF-stability}-{\rm (i)} shows that the DFE is globally asymptotically stable. If $N\in (\sum_{j\in\Omega}\gamma_j/\beta_j<(\sum_{j\in\Omega}\alpha_j\gamma_j)/(\sum_{j\in\Omega}\beta_i\alpha_j^2))$ and $\mathcal{R}_0<1$, model \eqref{model-mass-action} has multiple EE solutions. When limiting the movement of susceptible people, $N$ also plays an important role according to Theorem \ref{thm-ee-ds}:  the disease can be eliminated only when $N<\sum_{j\in\Omega}r_j$. This observation is consistent with the results in \cite{WuZou,castellano2022effect,li2023sis}.

When the disease eventually persists, it is critical to develop and implement adequate control strategies to alleviate the impact of the disease. One of the common disease control strategy is to restrict population movement, which was the case during the last global Covid-19 pandemic. One way to assess the effectiveness of such disease control strategies from a mathematical point of view is to study the profiles of the EE  with respect to small dispersal rates of the population.  In this direction, we first fixed the dispersal rate of the infected population and then examined the limit of EE as $d_S\to 0$. Our result in Theorem \ref{thm-ee-ds} indicates that if the total size of the population is below some critical number, then the disease can be significantly controlled by lowering the dispersal rate of the susceptible population. However, if the total population size exceeds this critical number, then disease may still persist even if the dispersal rate of the susceptible population is reduced. In fact, Proposition \ref{thm-ee-ds-2} confirms that the disease will persist if the total population size is large and the dispersal rate of the susceptible population is significantly small. Next, we fixed the diffusion rate of the susceptible population and then investigated the profiles of the EE  as $d_I\to 0$. Theorem \ref{thm-ee-di} suggests that if the dispersal rate of the susceptible population is small and the dispersal rate of the infected is significantly smaller, then the infected population will concentrate exactly on the  highest-risk patches.  Third, we restricted simultaneously the movement rates of both susceptible and infected people and  examined the profiles of the EE solutions in Theorem \ref{T6-ee-di-ds-small}. This result shows that the ratio of the dispersal rates play an important role on the dynamics of the disease.

\bibliographystyle{plain}
\bibliography{epidem}

\end{document}